\documentclass[twoside]{article}
\usepackage{eurosym}
\usepackage[utf8]{inputenc}
\usepackage{graphicx}
\usepackage{latexsym}
\usepackage{amsmath}
\usepackage{amsfonts}
\usepackage{amssymb}
\usepackage{hyperref}
\usepackage[usenames]{color}

\newcommand{\disp}{\displaystyle}

\newcommand{\cA}{{\mathcal A}}

\newcommand{\BB}{\mathbb{B}}
\newcommand{\cC}{{\mathcal C}}
\newcommand{\rC}{{\rm C}}
\newcommand{\rD}{{\rm D}}

\newcommand{\rF}{{\rm F}}

\newcommand{\cG}{{\mathcal G}}

\newcommand{\rI}{{\rm I}}

\newcommand{\rJ}{{\rm J}}
\newcommand{\rK}{{\rm K}}

\newcommand{\rL}{{\rm L}}

\newcommand{\rM}{{\rm M}}

\newcommand{\RR}{{\mathbb{R}}}

\newcommand{\rR}{{\rm R}}
\newcommand{\bS}{{\bf S}}
\newcommand{\rS}{{\rm S}}
\newcommand{\rT}{{\rm T}}
\newcommand{\XX}{{\mathbb{X}}}
\newcommand{\cX}{{\mathcal X}}

\newcommand{\rU}{{\rm U}}

\newcommand{\ru}{{\rm u}}

\newcommand{\hU}{\widehat{{\rm U}}}
\newcommand{\rV}{{\rm V}}

\newcommand{\rv}{{\rm v}}
\newcommand{\rW}{{\rm W}}

\newcommand{\rZ}{{\rm Z}}
\newcommand{\cZ}{{\mathcal Z}}
\newcommand{\fin}{\hfill\mbox{$\quad{}_{\Box}$}}
\newcommand{\fineq}{\vspace{-.75cm$\fin$}\par\bigskip}
\newcommand{\fineqnum}{\vspace{-.4cm$\fin$}\par\bigskip}

\newcommand{\n}[1]{\disp \|#1\|}
\newcommand{\nl}[1]{\disp \big \|#1\big \|}
\newcommand{\nll}[1]{\disp \bigg \|#1\bigg \|}
\newcommand{\nn}[1]{\disp \||#1 \||}
\newcommand{\nnl}[1]{\disp \big \||#1\big \||}

\newtheorem{theo}{\bf \sffamily Theorem}
\newtheorem{prop}{\bf \sffamily Proposition}
\newtheorem{coro}{\bf \sffamily Corollary}
\newtheorem{rem}{\bf \sffamily Remark}
\newtheorem{lemma}{\bf \sffamily Lemma}

\newtheorem{exam}{\bf \sffamily Example}

\begin{document}

\title{\Large \bfseries\sffamily  Nonlinear evolution equations with a non-Lipschitz perturbation: convergence of successive approximations and uniqueness of solutions}

\author{\bfseries\sffamily G. D\'{\i}az \& J.I. D\'{\i}az\thanks{JID was partially supported by the project PID-2020-112517GBI00 of the AEI and MCIU/AEI/10.13039/-501100011033/FEDER, EU. \hfil\break \indent {\sc
Keywords}:{nonlinear evolution equations, m-accretive operators,  Banach spaces, non-Lipschitz perturbation term, method of successive approximations, nonlinear Variation of Constants Formula}
 }}
\date{}
\maketitle

\begin{center}
	Dedicated to Roger Temam, always admired, on the occasion of his 85th birthday
\end{center}

\begin{abstract}
This paper investigates the existence and uniqueness of solutions for a nonlinear evolution equation governed by an m-accretive operator $\cA$ in a Banach space, presenting a perturbation term $\rF(t,\cdot)$ that does not satisfy the Lipschitz condition. Motivated by nonlinear diffusion models in climatology, we first establish the validity of the Variation of Constants Formula in this nonlinear framework, thereby reformulating the problem as a fixed-point problem for an integral operator. Under structural, boundedness, and unique continuation conditions on an associated scalar integral equation, we prove the convergence of the method of successive approximations towards a unique mild solution. This constructive approach extends previous uniqueness results for semilinear partial differential equations with non-Lipschitz perturbations to the more general setting of nonlinear operators in Banach spaces, which need not be reflexive. 

\end{abstract}

\section{Introduction}

One of the main motivations of the present work is the intention of the authors
to extend to the case of a quasilinear diffusion operator whose results (\cite%
{Diaz26stochastic}) dealing with a stochastic diffusive energy balance
climate model with a multiplicative noise modeling the Solar variability. In
that paper the discontinuous (or multivalued) co-albedo function was replaced
by a non-Lipschitz function $\beta (u)$ (since the considered noise was of
multiplicative type) which also allows us to identify the location of the polar
caps with a precision similar to that achieved by the discontinuous function
proposed by Budyko in 1969, but presents fewer difficulties in the treatment
of the stochastic framework. The reason for such non-linear diffusion of the
type $-((1-x^{2})|u^{\prime }|^{p-2}u^{\prime })^{\prime }+\epsilon
e(u)=\lambda f(u)$, with $p=3$, was supported by the paper of Stone \cite%
{Stone} (see the mathematical study in \cite{DiazEscorial}, and \cite%
{Bensid-D3})). In fact, a more general model in which the Earth is modelled
by a Riemannian manifold without boundary and a diffusion operator similar
to the usual $p$-Laplace operator $\Delta _{p}u,$ with $p>1$, was considered
in (\cite{DiazTello} and \cite{Diaz-Hetzer}). Here we will not present the
study in a stochastic framework (it could not be too difficult with the help
of the treatment of the stochastic $p$--Laplace operator made, for instance, in 
\cite{LiuRockner}). We will consider here a pure deterministic nonlinear
evolution problem stated in an abstract setting for a non-linear operator $\cA\ru(t)$ 
\begin{equation}
(P)\left\{ 
\begin{array}{l}
\dfrac{d\ru(t)}{dt}+\cA\ru(t)=\rF\big (t,%
\ru(t)\big ),\quad 0<t<\rT<\infty , \\[0.2cm]
\ru(0)=\ru_{0}.
\label{eq:abstract}
\end{array}%
\right. 
\end{equation}%
More precisely, we consider a Banach space $\cX$, not necessarily
reflexive, and assume that 
\begin{equation*}
\cA\text{ is a m-accretive operator on }\cX
\end{equation*}%
(to simplify the notation, we write the problem as for $\cA$
univalued, but all the results of this paper are valid for the general case
of $\cA$ multivalued). We also assume that 
\begin{equation*}
\overline{\mathrm{D}(\cA)}=\cX\text{ and }\ru%
_{0}\in \cX.
\end{equation*}%
On the perturbation term $\rF(t,\ru)$ we will assume the
following conditions:

\begin{description}
\item[\textbf{a1})] \textbf{Structural condition}. There exists a function $%
\mathrm{K}(t,\rU)\ge 0,~t\in [0,\rT],~\rU\ge 0,$
locally integrable function in $t$ for each fixed $\rU\ge 0$, continuous and
nondecreasing in $\rU$ for fixed $a.e.~t\in (0,\rT)$ and such $\rK(t,0)\equiv 0$ and 
\begin{equation}
\displaystyle \|\rF(t,\ru)-\rF (t,\rv)\|_{{%
\mathcal{X}}}\le \rK \big (t,\displaystyle \|\ru-\rv\|%
\big )_{\cX}\big )  \label{eq:a1}
\end{equation}
for $a.e.~t\in (0,\rT)$ and all $\ru,\rv\in \cX
$,

\item[\textbf{a2})] \textbf{Boundedness condition}. For a function $\mathrm{K%
}(t,\rU)$ as in \textbf{a1)} the integral equation 
\begin{equation}
\rU(t)=\rU_{0}+\int_{0}^{t}\rK\big (s,\rU(s)\big
)ds,\quad 0\leq t\leq \rT,  \label{eq:a2}
\end{equation}%
admits a scalar global solution $\rU(t)$ on $[0,\rT]$, where $%
\rU_{0}\geq 0$.

\item[\textbf{a3})] \textbf{Unique continuation condition}. The function $%
\rU(t)\equiv 0$ is the only non negative solution of 
\begin{equation}
\rU(t)\leq \int_{0}^{t}\rK\big (s,\rU(s)\big )%
ds,\quad 0\leq t\leq \rT.  
\label{eq:a3}
\end{equation}%
\label{desc:Taniguchiconditions}
\end{description}
(see \cite{Taniguchi1992} or \cite{Diaz26stochastic}). We will use the
following notion of solution: a function $\ru\in \cC\big ([0,\rT]:\cX\big )$ is said a \textit{mild solution}
of problem $(P)$ if $\rF\big (\cdot ,\ru(\cdot \big ))\in
\rL^{1}(0,\rT;\cX)$ (the usual Bochner integral) and 
\begin{equation}
\ru(t)=\rS(t)\ru_{0}+\int_{0}^{t}\rS(t-s)\rF\big (s,\ru(s\big )ds,\quad 0\leq t<\rT,
\label{eq:semilinearintegralequation0}
\end{equation}%
where $\{\rS(t)\}_{t\geq 0}$ is the semigroup of contractions
generated by $\cA.$
\par
We point out that expression (\ref{eq:semilinearintegralequation0}) comes
from the usual Constants Variations Formula (sometimes called the Variation
of parameters or Duhamel's principle) in the case in which operator $\cA$ is linear (see $e.g.$  \cite[Section 5.4 ]{BCP}). Nevertheless, we will prove here, it seems that by first time in the literature, that such formula also holds when operator $\cA$ is a nonlinear m-accretive operator with $\rD(\cA)=\cX$  (see Section 2). It is clear that the abstract framework allows to apply our results for many other applications different than the mentioned diffusive energy balance climate models (see, $e.g.$, the monographs \cite{Brezis}, \cite{Benilantesis}, \cite{BCP}, \cite{Barbu-libro}, \cite{DPitman} and \cite{TemamLibro}, to mention only some few of them).
\par
Let us mention that the existence of solutions for perturbed nonlinear Cauchy problems as $(P)$ was obtained in the literature even for more general assumptions than the assumed here: for instance for the case in which $\rF\big (\cdot ,\rU\big )$ is a multivalued operator (see, $e.g.$, \cite{DVrabie}, \cite{Vrabie}, \cite{DiazEscorial}, 
 \cite{DiazTello}, \cite{D-Monotonicity}). One of the main tools used in
 those studies is the Kakutani fixed point theorem for multivalued operators.
Moreover, it is also well known that when $\rF\big (\cdot ,\ru(\cdot \big ))$ is a multivalued operator the uniqueness of solution only
holds in the class of non-degenerate solutions (see \cite{Fereisel}, \cite{DiazEscorial}, \cite{DiazTello}). In fact the uniqueness of solutions for
quasilinear partial differential equations with some non-Lipschitz
perturbations were already studied in the very important paper \cite%
{FujitaWatanabe1968} (see also \cite{Cazenave-Escobedo}, \cite{D-Hernandez-Ilasov}). The uniqueness of the solution in the semilinear
PDEs under the above mentioned conditions were already considered in \cite{Yamada1981}, \cite{Taniguchi1992}, \cite{da1994stochastic}, \cite%
{BarbuBocsan2002} and \cite{Diaz26stochastic} in the stochastic framework for semilinear equations. 
\par
We want to go further than the paper \cite{FujitaWatanabe1968} where criteria
for the uniqueness and non-uniqueness of solutions are presented because we want
to get a constructive convergent successive approximations scheme for
the solutions, in this abstract framework, implying the uniqueness of
solutions in the absence of any Lipschitz condition on the perturbation term. We recall that for ordinary differential equations, the uniqueness of solutions can
still be obtained under weaker conditions, such as the classical Osgood
criterion \cite{Osgood1898}. In the context of partial differential
equations, the situation is significantly more subtle. The pioneering work
of Fujita and Watanabe \cite{FujitaWatanabe1968} showed that quasi-linear
parabolic equations may exhibit both uniqueness and non-uniqueness phenomena
depending on the growth behavior of the nonlinear term. Their results
highlight important differences between parabolic and elliptic problems. In particular, conditions ensuring uniqueness in elliptic equations, often based on sub- and supersolution techniques (see, $e.g.$, Amann \cite{Amann1972} and Pao \cite{Pao1992} for linear operators and Díaz-Saa \cite{DiazSaa} and \cite{D-Monotonicity} for the case of the p-Laplace operator), do not directly extend to the parabolic framework.

Our main argument is to search for mild solutions through fixed points $\ru$ of the operator $\cG^{\ru_{0}}$ 
$$ 
\ru=\cG^{\ru_{0}}\ru, 
$$
defined as 
$$
\cG^{\ru_{0}}\ru(t)\doteq \rS(t)\ru_{0}+\int_{0}^{t}\rS(t-s)\rF\big (s,\ru(s)\big )%
ds,\quad 0\leq t<\rT.  
$$
In order to prove the existence of the fixed point $\ru$ a key point
is to find some adequate topology in which the relative Picard type
approximations 
\begin{equation}
\ru_{n+1}=\cG^{\ru_{0}}\ru_{n},~n\geq 0,
\label{eq:Picardapproximations}
\end{equation}%
starting at $\ru_{0},$ converge to some function $\ru$. 

\begin{theo}
Let $\cA$ be a m-accretive operator on the Banach space $\cX$, such that $\overline{\rD(\cA)}=\cX$, and let $\ru_{0}\in \cX.$ Assume {\bf a1)}, {\bf a2)} and {\bf a3)}. Then the problem $(P)$ has a unique
mild solution $\ru\in\cC\big ([0,\rT]:\cX\big )$. Moreover, $\ru$ is the fixed point of operator $\cG^{\ru_{0}}$ and it can be constructed as the limit of the
Picard type approximation $\{\ru_{n}\}_{n\geq 0}\subset \cC\big ([0,\rT]:\cX\big )$ given by \eqref{eq:Picardapproximations}. Moreover, we have the
estimate 
\begin{equation}
\|\ru\|_{\cC\big ([0,t]:\cX\big )}\leq \rU(t),\quad \hbox{ for all }t\in [0,\rT],  \label{eq:estimatebX}
\end{equation} 
where $\rU(t)$ is the scalar function given as the solution of the integral equation 
\begin{equation}
\rU(t)=\rU_{0}+\int_{0}^{t}\mathrm{K}\big (s,\rU(s)\big
)ds,\quad t\in \lbrack 0,\rT],  \label{wq:functionu}
\end{equation}%
with $\rU_{0}=\displaystyle\|\ru_{0}\|_{\cX}$. In addition, the solution $\ru$ can be also identified as the unique limit of the implicit Euler scheme associated to the
nonhomogeneous equation with the forcing term $g(t)=\rF(t,\mathrm{u(t)%
})$ a.e. $t\in (0,\rT)$.
\label{teo:existence}
\end{theo}
In fact,  in the proof of Theorem  \ref {teo:existence} we will prove the inequality 
$$
\nl{\ru-\widehat{\ru}}_{\BB_{t}}\le \n{\ru_{0}-\widehat{\ru}_{0}}_{\cX}+\int^{t}_{0}\rK\big (s, \n{\ru-\widehat{\ru}}_{\BB_{s}}\big )ds,\quad t\in [0,\rT],
$$
where $\ru$ and $\widehat{\ru}$ are two eventual fixed points of $\cG^{\ru_{0}}$ and $\cG^{\widehat{\ru}_{0}}$, respectively,  in $\cC\big ([0,\rT]:\cX\big )$. For the particular case 
$$
\rK(t,\rU)=\phi(t)\vartheta(\rU),\quad t\in[0,\rT],~\rU\ge 0
$$
we have the  estimate
$$
\int^{\nl{\ru-\widehat{\ru}}_{\cC\big ([0,t]:\cX\big )}}_{\n{\ru_{0}-\widehat{\ru}_{0}}_{\cX}}\dfrac{ds}{\vartheta(s)}\le \int^{t}_{0}\phi(s)ds,\quad 0\le t\le \rT
$$
provided $\n{\ru_{0}-\widehat{\ru}_{0}}_{\cX}>0$ (see Corollay \ref{coro:uniquenessestimate} below). Moreover, the Osgood condition 
$$
\int_{0^{+}}\dfrac{d\rU}{\vartheta(\rU)}=\infty
$$
implies {\bf a3)}.

The organization of the rest of this paper is as follows. In Section \ref{sec:CVF} we prove that the classical Constant Variations Formula remains
valid beyond the linear framework: we will prove it for a possible nonlinear
operator m-accretive operator $\cA$ with dense domain on $\cX$. \ Section \ref{sec:exisuni} is devoted to the proof of
Theorem \ref{teo:existence}. 
Finally, in Section \ref{sec:integralequation}
we study the integral equation \eqref{eq:a2} and give several examples of
the kernel $\rK\big (s,\rU\big )$ which satisfy the unique
continuation condition {\bf a3)}. In particular, we apply the
above Theorem \ref{teo:existence} to a relevant illustration (see Examples \ref{exam:beta} and \ref{exam:beta2} below) introduced in \cite{Diaz26stochastic} in studying
Stochastic Energy Balance Climate models. 
Here \eqref{eq:a2} becomes 
\begin{equation}
\rU(t)=\rU_{0}+\rS_{0}\int_{0}^{t}\beta \big (\rU(s)\big )ds,\quad 0\leq t\leq \rT,  \label{eq:a2beta}
\end{equation}%
where $\bS_{0}$ is positive insulation constant and  $\beta (s)$ is a co-albedo profile in the hybrid model of \cite{Diaz26stochastic}.

\section{Extension of the Constant Variation Formula to nonlinear operators on Banach spaces}
\label{sec:CVF}
Let $\cX$ be a Banach space.
Let 
$
\cA:\rD(\cA)\subset \cX \to 2^\cX
$
be an $m$-accretive operator, $i.e.$ such that
$$
\left \{
\begin{array}{c}
\|(x-\widehat{x})+\lambda (y-\widehat{y})\|\ge \|x-\widehat{x}\|,\quad (x,y),~(\widehat{x},\widehat{y})\in\cA),~\lambda \ge 0\\ [.2cm]
\hbox{Rank }( \rI+\cA)=\XX,
\end{array}
\right .
$$
hold. 
We assume throughout that
$$
\overline{\rD(\cA)} = \cX.
$$
Let $f \in L^1(0,T;X)$. We consider the evolution problem

$$
\begin{cases}
u^{\prime }(t)+\cA u(t)\ni f(t), \\ 
u(0)=u_{0}.%
\end{cases}
$$
For $\lambda >0$ define the resolvent
$$
\rJ_{\lambda }=(\rI+\lambda \cA)^{-1}.
$$
Since $\cA$ is $m$-accretive,
$$
\rJ_{\lambda }:\cX\rightarrow \rD(\cA)
$$%
is well defined and nonexpansive:
$$
\|\rJ_{\lambda }x-\rJ_{\lambda }y\|\leq \|x-y\|.
$$
The nonlinear semigroup generated by $\cA$ is defined by
$$
\rS(t)x=\lim_{n\rightarrow \infty }\left( \rI+\frac{t}{n}\cA\right) ^{-n}x.
$$
The Crandall--Liggett Theorem (\cite{CrandallLiggett}) guarantees that the limit exists for every $
x\in \cX$ and defines a contraction semigroup. Moreover, for the
non-homogeneous problem we can use the implicit Euler scheme to solve the
problem. Let
$
\lambda =\frac{T}{n},~t_{k}=k\lambda .
$
From $u_{0}$, we consider the approximation
$$
u_{k+1}=\rJ_{\lambda }(u_{k}+\lambda f(t_{k})) 
$$
Define the piecewise constant interpolation
$$
u_{\lambda }(t)=u_{k},\quad t\in [t_{k},t_{k+1}).
$$%
By the Crandall--Liggett theorem we know the convergence of $u_{\lambda }(t)$
to the unique solution $u(t)$ of the problem. Our goal is to extend the
representation formula
$$
u(t)=\rS(t)u_{0}+\int_{0}^{t}\rS(t-s)f(s)\,ds
$$%
which is well-known for linear operators, to the case of nonlinear operators 
$\cA$, where now $\rS(t)$ is the nonlinear semigroup generated by $\cA$.

\begin{theo}
Let $\cA$ be an $m$-accretive operator on a Banach space $\cX$ with dense domain. Let $f\in \rL^{1}(0,T;\cX)$. Then the associated implicit Euler scheme converges to a limit function $u(t)$ which  satisfies
$$
u(t)=\rS(t)u_{0}+\int_{0}^{t}\rS(t-s)f(s)\,ds.
$$
\label{theo:NCVF}
\end{theo}
\vspace*{-.5cm}

The proof will be a consequence of several general properties. We start with
the comparison with the homogeneous scheme. Let
$$
v_{k+1}=\rJ_{\lambda }v_{k},\qquad v_{0}=u_{0}.
$$%
Then
$
v_{k}=\rJ_{\lambda }^{k}u_{0}.
$
Recall that by the Crandall--Liggett Theorem (\cite{CrandallLiggett}), $v_{k}\rightarrow \rS(t_{k})u_{0}.$
\begin{lemma}
The following estimate holds
$$
\|u_{k+1}-v_{k+1}\| \le \|u_k-v_k\| + \lambda \|f(t_k)\|.
$$
\label{lemma:firsestimate}
\end{lemma}
\noindent {\sc Proof.} Using the nonexpansiveness of $\rJ_\lambda$,
$$
\begin{array}{ll}
\|u_{k+1}-v_{k+1}\| &\hspace*{-.2cm} =
\|\rJ_\lambda(u_k+\lambda f(t_k)) - \rJ_\lambda v_k\|\\ [.1cm ]
&\hspace*{-.2cm} 
\le \|u_k+\lambda f(t_k) - v_k\|\\ [.1cm ]
& \hspace*{-.2cm}  \le
\|u_k-v_k\| + \lambda \|f(t_k)\|.
\end{array}
$$
\fineq
\par
\noindent 
Iterating this inequality gives
$$
\|u_k\|-\|v_k\| \le \|u_k-v_k\| \le \sum_{i=0}^{k-1}\lambda \|f(x_{i})\|.
$$
It will be useful to prove the boundedness of $u_{k}.$

\begin{lemma}
The sequence $\{u_k\}_{k}$ is bounded.
\label{lemma:boundedsequence}
\end{lemma}
\par
\noindent {\sc Proof.} Using the previous estimate,
$$
\|u_k\|
\le
\|v_k\| + \sum_{\rJ=0}^{k-1}\lambda \|f(t_\rJ)\|.
$$
Since $\rS(t)$ is a contraction, 
$$
\|v_k\| \le \|u_0\|.
$$
Thus
$$
\|u_k\|
\le
\|u_0\| + \int_0^T \|f(s)\| ds .
$$
\fineq

\par
Let us prove now a consequence of the a priori bounds of the implicit Euler scheme and the nonexpansiveness of the resolvent: a property that is consistent with the classical proof of the Crandall--Liggett Theorem (see (\cite{CrandallLiggett})).

\begin{lemma}\label{lemma:LemmDetails} 
There exists a constant $\rC>0$, independent of $\lambda$, such that the following estimate holds
\begin{equation}
\|u-\rJ_{\lambda}u\|\le \lambda \rC,
\label{eq:resolvent_estimate_C}
\end{equation}
\end{lemma}

\noindent {\sc Proof.} From the definition of the resolvent we have
$$
u-\rJ_{\lambda}u \in \lambda \cA(\rJ_{\lambda}u).
$$
Thus, there exists $w_\lambda\in \cA(\rJ_{\lambda}u)$ such that
$$
u-\rJ_{\lambda}u=\lambda w_\lambda
\quad\text{and hence}\quad
\|u-\rJ_{\lambda}u\|=\lambda\|w_\lambda\|.
$$
Therefore, it suffices to obtain a bound on $\|w_\lambda\|$ independent of $\lambda$. To this end, we use the implicit Euler scheme
$$
u_{k+1}=\rJ_\lambda(u_k+\lambda f(t_k)).
$$
From Lemma \ref{lemma:boundedsequence}, the sequence $\{u_k\}_k$ is bounded:
$$
\|u_k\|\le \rM\doteq \|u_0\|+\int_0^T \|f(s)\|\,ds.
$$
Since $\rJ_\lambda$ is nonexpansive, it follows that
$$
\|\rJ_\lambda(u_k+\lambda f(t_k))\|\le \|u_k+\lambda f(t_k)\|
\le \rM+\lambda\|f(t_k)\|,
$$
so the sequence $\{u_{k+1}\}_k$ also remains in a bounded subset of $\cX$, uniformly with respect to $\lambda$. Moreover, from the scheme we have
$$
u_k+\lambda f(t_k)-u_{k+1}\in \lambda \cA(u_{k+1}),
$$
so there exists $z_k\in \cA(u_{k+1})$ such that
$$
u_k+\lambda f(t_k)-u_{k+1}=\lambda z_k.
$$
Hence,
$$
\|z_k\|\le \frac{\|u_k-u_{k+1}\|}{\lambda}+\|f(t_k)\|.
$$
Using again, the contractivity of $\rJ_\lambda$,
$$
\|u_{k+1}-u_k\|
=\|\rJ_\lambda(u_k+\lambda f(t_k))-\rJ_\lambda u_k\|
\le \lambda\|f(t_k)\|,
$$
and therefore
$$
\|z_k\|\le 2\|f(t_k)\|.
$$
\noindent
This shows that the elements of $\cA(u_{k+1})$ selected by the scheme are uniformly bounded in terms of $\|f\|_{L^1}$, independently on $\lambda$. Since $\rJ_\lambda u$ belongs to the same bounded region (by nonexpansiveness and the boundedness of the data), we conclude that there exists a constant $\rC>0$, independent of $\lambda$, such that
$$
\|w_\lambda\|\le \rC.
$$
This proves \eqref{eq:resolvent_estimate_C}.$\fin$

\par
\medskip

Now, let us prove a discrete and nonlinear version of the so-called
Variation of Constants Formula (in the framework of linear operators)

\begin{prop}
We have 
$$
u_{k}=\rJ_{\lambda }^{k}u_{0}+\sum_{i=0}^{k-1}\rJ_{\lambda }^{k-i-1}\lambda
f(t_{i})+r_{k}
$$%
with $\|r_{k}\|\rightarrow 0.$
\end{prop}
\par
\noindent {\sc Proof.} Let
$$
u_{k+1}=\rJ_{\lambda }(u_{k}+\lambda f(t_{k})),\qquad u_{0}\in X.
$$%
Let the homogeneous scheme be
$$
v_{k+1}=\rJ_{\lambda }v_{k},\qquad v_{0}=u_{0}.
$$%
Then $v_{k}=\rJ_{\lambda }^{k}u_{0}.$ Define
$$
w_{k}=u_{k}-v_{k}.
$$%
Using the recursion we have
$$
w_{k+1}=\rJ_{\lambda }(u_{k}+\lambda f(t_{k}))-\rJ_{\lambda }v_{k}.
$$%
Introduce
$$
z_{k,\rJ}\doteq \rJ_{\lambda }^{k-i-1}\left( u_{i+1}-\rJ_{\lambda }u_{i}\right)
$$%
and observe that
$$
u_{k}-v_{k}=\sum_{i=0}^{k-1}\rJ_{\lambda }^{k-i-1}\left( u_{i+1}-\rJ_{\lambda
}u_{i}\right) .
$$%
Indeed this follows from the telescoping identity
$$
u_{k}=\rJ_{\lambda }^{k}u_{0}+\sum_{i=0}^{k-1}\rJ_{\lambda }^{k-i-1}\big(%
u_{i+1}-\rJ_{\lambda }u_{i}\big).
$$%
From the scheme
$$
u_{\rJ+1}=\rJ_{\lambda }(u_{i}+\lambda f(t_{i}))
$$%
we define the local error
$$
e_{i}=u_{i+1}-\rJ_{\lambda }u_{i}.
$$%
Thus
$$
u_{k}=\rJ_{\lambda }^{k}u_{0}+\sum_{i=0}^{k-1}\rJ_{\lambda }^{k-i-1}e_{i}.
$$%
The key point here is that
$$
e_{i}=\rJ_{\lambda }(u_{i}+\lambda f(t_{i}))-\rJ_{\lambda }u_{i}.
$$%
Since $\rJ_{\lambda }$ is nonexpansive,
$$
\|e_{i}\|\leq \lambda \|f(t_{i})\|.
$$%
Denote $
g_{i}\doteq\lambda f(t_{i}),
$
and consider the approximation
$
\disp \sum_{i=0}^{k-1}\rJ_{\lambda }^{k-i-1}g_{i}.
$
We compare this with the true discrete contribution
$$
\sum_{i=0}^{k-1}\rJ_{\lambda }^{k-i-1}e_{i}.
$$%
Define the difference
$$
r_{k}=\sum_{i=0}^{k-1}\rJ_{\lambda }^{k-i-1}(e_{i}-g_{i}).
$$%
This is the nonlinear error term. Using the nonexpansiveness of $\rJ_{\lambda
} $ we obtain
$$
\|r_{k}\|\leq \sum_{i=0}^{k-1}\|e_{i}-g_{i}\|.
$$%
Now
$$
e_{i}-g_{i}=\rJ_{\lambda }(u_{i}+\lambda f(t_{i}))-\rJ_{\lambda }u_{i}-\lambda
f(t_{i}).
$$%
Adding and subtracting $\rJ_{\lambda }(u_{i})+\lambda f(t_{i})$ we get
$$
e_{i}-g_{i}=\big(\rJ_{\lambda }(u_{i}+\lambda f(t_{i}))-\rJ_{\lambda }u_{i}\big)%
-\lambda f(t_{i}).
$$%
Using the Lipschitz property of $\rJ_{\lambda }$
$$
\|\rJ_{\lambda }(u_{i}+\lambda f)-\rJ_{\lambda }u_{i}\|\leq \lambda
\|f\|.
$$%
Therefore
$$
\|e_{i}-g_{i}\|\leq \|\rJ_{\lambda }(u_{i}+\lambda
f(t_{i}))-(u_{i}+\lambda f(t_{i}))\|+\|\rJ_{\lambda }u_{i}-u_{i}\|.
$$%
By the resolvent identity, we know that
$$
u-\rJ_{\lambda }u\in \lambda \cA(\rJ_{\lambda }u).
$$%
\noindent
By Lemma \ref{lemma:LemmDetails}, one knows that
\[
\|u-\rJ_{\lambda }u\|\leq \lambda \rC,
\]
for some constant $\rC>0$ independent of $\lambda$. Since the sequence $u_{i}$ is bounded, we obtain
$$
\|e_{i}-g_{i}\|\leq \rC\lambda ^{2}.
$$%
Thus
$$
\|r_{k}\|\leq \rC\sum_{i=0}^{k-1}\lambda ^{2}=\rC k\lambda ^{2}.
$$
But
$
k\lambda =t_{k}\leq T
$, 
hence
$$
\|r_{k}\|\leq \rC T\lambda ,
$$
and therefore
$
r_{k}\rightarrow 0~(\lambda \rightarrow 0).\fin$
\par
\medskip
\noindent {\sc Proof of the Theorem \ref{theo:NCVF})}. We proceed in two steps. \textit{Step 1: The case} $f\in \cC([0,\rT];\cX)$. Let $t_{k}\rightarrow t.$ By the Crandall--Liggett Theorem (\cite{CrandallLiggett}), $%
\rJ_{\lambda }^{k}u_{0}\rightarrow \rS(t)u_{0}.$ Moreover,
$$
\rJ_{\lambda }^{k-i-1}x\rightarrow \rS(t-t_{i})x.
$$%
Therefore
$$
\sum_{i=0}^{k-1}\lambda \rJ_{\lambda }^{k-i-1}f(t_{i})\rightarrow
\int_{0}^{t}\rS(t-s)f(s)\,ds.
$$%
This is a Riemann sum for the Bochner integral. Combining the convergence of
the homogeneous part and the Riemann convergence of the forcing term gives
the desired result.
\par
\noindent {\em Step 2: Extension to} $f\in \rL^{1}(0,\rT;\cX).$
Choose a sequence $\{f_{n}\}_{n}\subset \cC([0,\rT];\cX)$ such that $%
f_{n}\rightarrow f$ in $\rL^{1}$. For each $n$, let $u_{k}^{(n)}$ be the
scheme with forcing $f_{n}$, and let $u^{(n)}(t)$ be the limit from Step 1
$$
u^{(n)}(t)=\rS(t)u_{0}+\int_{0}^{t}\rS(t-s)f_{n}(s)\,ds.
$$%
By the discrete Duhamel estimate, 
$$
\|u_{k}-u_{k}^{(n)}\|\leq \sum_{i=0}^{k-1}\lambda \|
f(t_{i})-f_{n}(t_{i})\|.
$$%
As $\lambda \rightarrow 0$, the right-hand side converges to $\disp
\int_{0}^{t}\|f(s)-f_{n}(s)\|\,ds$. Hence, 
$$
\limsup_{\lambda \rightarrow 0}\|u_{k}-u^{(n)}(t)\|\leq
\limsup_{\lambda \rightarrow 0}\big \|u_{k}-u_{k}^{(n)}\big \|
+\limsup_{\lambda \rightarrow 0}\big \|u_{k}^{(n)}-u^{(n)}(t)\big \|\leq
\int_{0}^{t}\|f(s)-f_{n}(s)\|\,ds.
$$%
Taking $n\rightarrow \infty $, the right-hand side tends to $0$, so 
$$
u_{k}\rightarrow u(t)\doteq\lim_{n\rightarrow \infty }u^{(n)}(t).
$$%
Finally, we identify $u(t)$. By the continuous dependence estimate for mild
solutions, 
$$
\left \|u^{(n)}(t)-\left (\rS(t)u_{0}+\int_{0}^{t}\rS(t-s)f(s)\,ds\right )\right \|\leq
\int_{0}^{t}\|f_{n}(s)-f(s)\|\,ds\rightarrow 0.
$$%
Thus, 
$$
u(t)=\rS(t)u_{0}+\int_{0}^{t}\rS(t-s)f(s)\,ds.
$$
\fineq

\begin{rem}\rm
Notice that this nonlinear version of the Constant Variations Formula is
different to the version given in \cite{Alekseev} for nonlinear ordinary
differential equations (see also \cite{CasalDDVegas}).
\end{rem}

\section{Existence and uniqueness of solutions of problem (P)}
\label{sec:exisuni}
\par
As it was pointed out above, we study the general Cauchy problem governed by a nonlinear equation
$$
\left \{
\begin{array}{l}
\dfrac{d \ru(t)}{dt}+\cA \ru(t)=\rF\big (t,\ru(t)\big ),\quad 0<t<\rT<\infty,\\ [.2cm]
\ru (0)=\ru_{0},
\end{array}
\right .
$$
(see \eqref{eq:abstract} above) on a Banach space $\cX$ where we are considering function $\ru:[0,\rT]\rightarrow \cX$. In fact, we want to solve \eqref{eq:abstract} in the functional space $\cC \big ([0,\rT]:\cX\big ) $.
\par
Here  for term $\rF\big (t,\ru\big )\in\cX,~(t,\ru) \in [0,\rT]\times \cX$ we assume that the Bochner integral
$$
\int^{\rT}_{0}\rF(s,\ru)ds 
$$
is posed in $\cX$ as well as the condition {\bf a1)}, {\bf a2)} and {\bf a3)} (see Introduction). Finally, we assume that  $\cA $ is a m-accretive operator on $\cX$ generating a contraction  semigroup $\{\rS(t)\}_{t\ge 0}$ in $\overline{\rD(\cA)}\subset\cX$. Assume $\overline{\rD(\cA)}=\cX$. From the Constant Variations  Method we want to solve \eqref {eq:abstract} by {\em mild solutions}, thus by  functions $\ru\in \cC \big ([0,\rT]:\cX\big ) $  satisfying 
\begin{equation}
\ru(t)=\rS(t)\ru_{0} +\int_{0}^{t}\rS(t-s)\rF\big (s,\ru(s\big )ds,\quad 0\le t<\rT,
\label{eq:semilinearintegralequation}
\end{equation}
where $\ru_{0}\in \cX$.   So that, we search  mild solutions by means of the existence of fixed points  $\ru$ of $\cG^{\ru_{0}}$
\begin{equation}
\ru=\cG^{\ru_{0}}\ru,
	\label{eq:fixedpoint}
\end{equation}
for the integral operator
\begin{equation}
	\cG ^{\ru_{0}}\ru(t)\doteq \rS(t)\ru_{0} +\int_{0}^{t}\rS(t-s)\rF\big (s,\ru(s)\big )ds,\quad 0\le t<\rT.
	\label{eq:functionaloperatorG}
\end{equation}
The goal is to solve the fixed-point problem in the space $\cC\big ([0,\rT]:\cX\big)$. By simplicity we denote   $ \BB_{\rT}\doteq \cC\big ([0,\rT]:\cX\big )$ 
equipped with the norm
$$
\n{\ru}_{\BB_{\rT}}\doteq \sup_{0\le s\le \rT}\n{\ru(s)}_{\cX}.
$$
\par
\noindent
As it follows from the notations, the function $\ru \in\BB_{\rT}$ is evaluated in the Banach space $\cX$ and $\rU \in\cC([0,\rT]:\RR)$  is an scalar function.

\begin{rem}\rm  In considering assumptions on the term $\rF(s,\rU)$ we may pick up the choice
$$
\rK(t,\rU)=\phi(t)\vartheta(\rU),\quad t\in[0,\rT],~\rU\ge 0,
$$
where $\vartheta:~]0,\infty[\rightarrow ]0,\infty[$ is a continuous and increasing function verifying $\vartheta(0^{+})=0$ and $\phi$ is an integrable and non negative function on $[0,\rT]$. In this case, the assumption {\bf a2}) is immedate (see Remark \ref{rem:a22}). 
We emphasize that equality in \eqref{eq:a2} cannot be substituted by an inequality (see the proof of Theorem \ref{teo:existence} and Remark \ref{rem:a21}). 
\par
\noindent The condition {\bf a3}) is studied in Section \ref{sec:integralequation}). As it is well known, in the Lipschitz case
$$
\vartheta(\rU)=\rU,\quad t\in[0,\rT],~\rU\ge 0,
$$
 the condition {\bf a3}) follows from the Gronwall Inequality (see Corollary \ref{coro:Gronwallinequality}).$\fin$
\label{rem:separatedvariable}
\end{rem}
\par
In Section \ref{sec:integralequation} below, we pick up some general comments in which 
the boundedness condition {\bf a2}) and the unique continuation criterion {\bf a3}) hold. 

\begin{rem}\em In order to provide inequalities as \eqref{eq:a1} we may consider real functions $\vartheta:[0,\infty[\rightarrow [0,\infty[$ satisfying $\vartheta(0)\ge 0$ with  the subadditive property
\begin{equation}
\vartheta(\rU)+\vartheta(\rW)\ge \vartheta(\rU+\rW)\quad \Leftrightarrow \quad \vartheta(\rU+\rW)-\vartheta(\rU)\le \vartheta(\rW),\quad \rU,\rW\ge 0.
\label{eq:subadditive}
\end{equation}
So that, let $\rU,\rV\ge 0$ such that $\rV\ge \rU$ form $\rW=\rV-\rU\ge 0$ for which
$$
0\le \vartheta(\rV)-\vartheta(\rU)=\vartheta(\rW+\rU)-\vartheta(\rU)\le \vartheta(\rW)=\vartheta(\rV-\rU),
$$
whenever $\vartheta $ is nondecreasing. By means of a similar reasoning, we conclude 
\begin{equation}
|\vartheta (\rU)-\vartheta (\rV)|\le \vartheta (|\rU-\rV|),\quad \rU,\rV\ge 0,
\label{eq:continuitymodulus}
\end{equation}
provided $\vartheta $ is nondecreasing. In this case, under \eqref{eq:subadditive}, the condition {\bf a1)} holds for the choice $\rF (t,\rU)=\rK (t,\rU)=\phi (t) \vartheta (\rU),~t\in [0,\rT], \rU\ge 0.\fin$
\label{rem:continuitymodulus}
\end{rem}
\begin{rem}\em The subaditive property \eqref{eq:subadditive} is satisfied by real concave functions $\vartheta:[0,\infty[\rightarrow [0,\infty[$ satisfying $\vartheta(0)\ge 0$. Indeed, it follows
$$
\vartheta(\lambda \rZ)=\vartheta \big (\lambda \rZ+(1-\lambda )0\big )\ge \lambda \vartheta(\rZ)+(1-\lambda )\vartheta(0)\ge  \lambda \vartheta(\rZ),\quad z\ge 0,~0<\lambda<1.
$$ 
In particular, given $\rU,\rW>0$ by choosing $\lambda _{\rU,\rW}=\dfrac{\rU}{\rU+\rW} \in ]0,1[$ it follows
$$
\left \{
\begin{array}{l}
	\vartheta(\rU)=\vartheta\big (\lambda _{\rU,\rW}(\rU+\rW)\big )\ge \lambda _{\rU,\rW}\vartheta(\rU+\rW),\\
	\vartheta(\rW)=\vartheta\big ((1-\lambda _{\rU,\rW})(\rU+\rW)\big )\ge (1-\lambda _{\rU,\rW})\vartheta(\rU+\rW),
\end{array}
\right .
$$
whence one concludes the subadditive property
$$
\vartheta(\rU)+\vartheta(\rW)\ge \vartheta(\rU+\rW).
$$
Arguing as in \cite[Lemma 4.3]{Diazlarge2023}, we may obtain the subadditive property \eqref{eq:subadditive} by transfer without concavity settings. Indeed, assume
$$
q(\rU)+q(\rV)\ge q(\rU+\rV),\quad \rU,\rV\ge 0
$$
and
\begin{equation}
\dfrac{\vartheta(\rU)}{q(\rU)}\quad \hbox{is non increasing},
\label{eq:subadditivetransfer}
\end{equation}
provided $q(\rU)>0,~\rU>0$. Then
$$
\vartheta(\rU)+\vartheta(\rV)=q(\rU)\dfrac{\vartheta(\rU)}{q(\rU)}+q(\rV)\dfrac{\vartheta(\rV)}{q(\rV)}\ge (q(\rU)+q(\rV))\dfrac{\vartheta(\rU+\rV)}{q(\rU+\rV)} \ge \vartheta(\rU+\rV),
$$
thus, the  subadditivity of the function $q(\rU)$ is transferred to the function $\vartheta(\rU)$ provided \eqref{eq:subadditivetransfer}. In particular, any function $\vartheta$ such that
$$
\dfrac{\vartheta(\rU)}{\rU^{m}}\quad \hbox{is non increasing}
$$
for some $0<m\le 1$ is subadditive and the inequality \eqref{eq:continuitymodulus}
holds whenever $\vartheta $ is nondecreasing.$\fin$
\label{rem:subadditiveproperty}
\end{rem}

\begin{exam}\rm 
Let us consider the decreasind function
\begin{equation}
\vartheta(\rU)=\rU\ln \rU,\quad 0\le \rU \le e^{-1}.
\label{eq:modulocoalbedo}
\end{equation}
From Remark \ref{rem:subadditiveproperty} one proves that $\vartheta $ satisfies the inequality
$$
|\vartheta(\rU)-\vartheta(\rV)|\le \vartheta(|\rU-\rV|),\quad \rU,\rV\ge 0.
$$
Moreover, from \eqref{eq:mainexampleOsgood} below the function $\vartheta$ satisfies \eqref{eq:Osgoodcondition}.
Then we come back to the co-albedo profile 
\begin{equation}
\beta(\rU)=\left\{ 
\begin{array}{ll}
	\beta _{i}, & \quad \hbox{if $\rU<0$,} \\[0.1cm] 
	\dfrac{\beta _{w}-\beta_{i}}{\delta \ln \delta }\vartheta(\rU)+\beta_{i}, & \quad \hbox{if $0\le \rU\le \delta $}, \\[0.1cm] 
	\beta _{w}, & \quad \hbox{if $\rU>\delta$},
\end{array}
\right.  
\label{eq:coalbedoprofile}
\end{equation}
with $0<\delta<e^{-1}$ that governs the co-albedo function introduced in \cite{Diaz26stochastic} (see Figure 	\ref{fig:coalbedoprofile}). 
\begin{figure}[htp]
\begin{center}
\vspace*{.5cm}
\includegraphics[width=10cm]{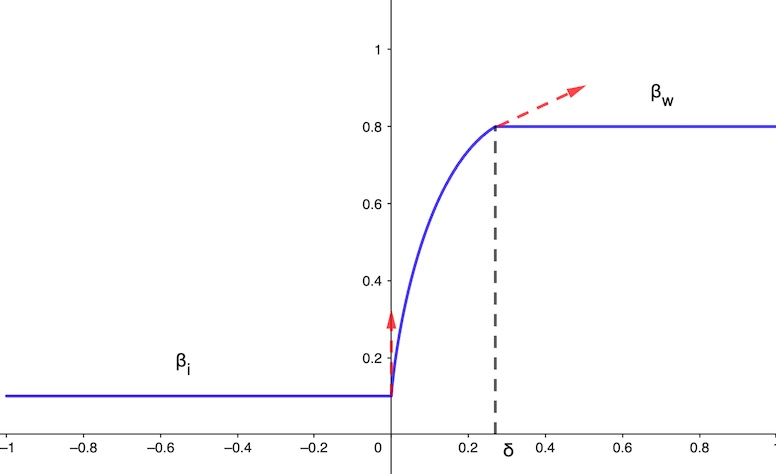}\\ 
\caption{Co-albedo profile (see \eqref{eq:coalbedoprofile})}
\label{fig:coalbedoprofile}
\end{center}
\end{figure}
We note that $\beta\in\cC(\RR)\cap \cC^{1}\big (\RR\setminus\{0,\delta\}\big )$, with
$$
\left\{
\begin{array}{l}
	\beta'(0^{-})=0\quad \hbox{and \quad }\beta'(0^{+})=+\infty,\\ [.1cm]
	\beta'\big (\delta^{-}\big )>0\quad \hbox{and} \quad \beta'(\delta^{+})=0.
\end{array}
\right .
$$
We claim that the profile $\beta $ verifies
$$
|\beta(\rU)-\beta(\rV)|\le \vartheta(|\rU-\rV|),\quad \rU,\rV\in\RR.
$$
Indeed, when $\rV\le 0$ one has
$$
\beta (\rU)-\beta (\rV)=\vartheta(\rU)-\vartheta(0)=\vartheta(\rU)\le  \vartheta(\rU-\rV) ,\quad \rV\le 0\le \rU\le \delta.
$$ 
Analogously, if $\rU\ge \delta$ one has
$$
\beta (\rU)-\beta (\rV)=\beta_{w}-\dfrac{\beta _{w}-\beta_{i}}{\delta \ln \delta }\vartheta(\rV)-\beta_{i}\le 0\le   \vartheta(\rU-\rV) ,\quad \rU\ge \delta\ge \rV\ge 0.
$$ 
Finally, the claim follows from
$$
\left \{
\begin{array}{l}
\beta (\rU)-\beta (\rV)=\vartheta(\rU)-\vartheta(\rV)\le \vartheta(\rU-\rV),\quad \delta \ge \rU\ge \rV\ge 0,\\ [.2cm]
\beta (\rU)-\beta (\rV)=\vartheta(\delta)-\vartheta(0)=\vartheta (\delta)\le \vartheta (\rU-\rV),\quad \rU\ge \delta \ge 0\ge \rV.
\end{array}
\right .
$$
Thus, for the choice $\rF(t,\rU)=\bS_{0}\beta (\rU),~t \in [0,\rT],~\rU\in\RR,$ the condition {\bf a1)} holds for the function $\rK(t,\rU)=\bS_{0}\vartheta (\rU),~t \in [0,\rT],~\rU\ge 0$.
Here $\bS_{0}$ is a positive constant. We come back to this choice in the Example \ref{exam:beta2} below. $\fin$
\label{exam:beta}
\end{exam}

\begin{prop}Assuming {\bf a1}{\rm )}, the operator $\cG^{\ru_{0}}:~\BB_{\rT}\rightarrow \BB_{\rT}$ given by  \eqref{eq:functionaloperatorG} is well defined. Moreover, one has the estimate
\begin{equation}
\nl{\cG^{\ru_{0}}\ru-\cG^{\widehat{\ru}_{0}}\rv}_{\BB_{t}}\le \n{\ru_{0}-\widehat{\ru}_{0}}_{\cX}+\int^{t}_{0}\rK\big (s, \n{\ru-\rv}_{\BB_{s}}\big )ds,\quad t\in [0,\rT].
 \label{eq:Gcontinuous}
\end{equation}
that implies the continuity of $\cG^{\ru_{0}}$. 
 \label{prop:Gcontinuous}
\end{prop}
{\sc Proof.}  Let $\ru\in \BB_{\rT}$. Then
$$
\begin{array}{ll}
\disp \sup_{0\le t\le \rT}\nl{\cG^{\ru_{0}}\ru(t)}_{\cX}
&\disp \hspace*{-.2cm} \le \sup_{0\le t\le \rT}\nl{\rS(t)\ru_{0}}_{\cX}+
\sup_{0\le t\le \rT}\nll{\int^{t}_{0}\rS(t-s)\rF\big (s,\ru(s)\big )}_{\cX}ds \\ [.45cm]
& \disp \hspace*{-.2cm} \le \n{\ru_{0}}_{\cX}+\sup_{0\le t\le \rT}
\int^{t}_{0}\n{\rF\big (s,\ru(s)\big )}_{\cX}ds  \\ [.45cm]
& \disp \hspace*{-.2cm} \le \n{\ru_{0}}_{\cX}+
\int^{\rT}_{0}\rK\big (s,\n{\ru(s)}_{\cX}\big ) ds \\ [.45cm]
& \disp \hspace*{-.2cm} \le \n{\ru_{0}}_{\cX}+
\int^{\rT}_{0}\rK\big (s,\n{\ru}_{\BB_{\rT}}\big ) ds<\infty,
\end{array}
$$
thus
\begin{equation}
\nl{\cG^{\ru_{0}}\ru}_{\BB_{\rT}}\le \n{\ru_{0}}_{\cX}+
\int^{\rT}_{0}\rK\big (s,\n{\ru}_{\BB_{\rT}}\big ) ds<\infty,
\label{eq:boundG}
\end{equation}
whence $\nl{\cG^{\ru_{0}}\ru}_{\BB_{\rT}}<\infty$ and $\cG^{\ru_{0}}\ru\in\BB_{\rT}$.
On the other hand, let $\ru,\rv\in \BB_{\rT}$. Then, by reasoning as above, and using the fact that the semigroup is of contractions,  we obtain
\begin{equation}
\begin{array}{ll}
\disp \sup_{0\le t\le \rT}\nl{\cG^{\ru_{0}}\ru(t)-\cG^{\widehat{\ru}_{0}}\rv(t)}_{\cX}&\disp \hspace*{-.2cm}\le \n{\ru_{0}-\widehat{\ru}_{0}}_{\cX}+\int^{\rT}_{0}\n{\rF\big (s,\ru(s)\big )-\rF\big (s,\rv(s)\big )}_{\cX}ds\\
&\disp \hspace*{-.2cm}\le \n{\ru_{0}-\widehat{\ru}_{0}}_{\cX}+\int^{\rT}_{0}\rK\big (s, \n{\ru-\rv}_{\BB_{\rT}}\big )ds,
\end{array}
\label{eq:Kinequality0}
\end{equation}
briefly
\begin{equation}
\nl{\cG^{\ru_{0}}\ru-\cG^{\widehat{\ru}_{0}}\rv}_{\BB_{\rT}}\le  \n{\ru_{0}-\widehat{\ru}_{0}}_{\cX}+\int^{\rT}_{0}\rK\big (s, \n{\ru-\rv}_{\BB_{\rT}}\big )ds.
\label{eq:Kinequality}
\end{equation}
In particular the continuity of $\cG^{\ru_{0}}$ on $\BB_{\rT}$ follows.$\fin$

\begin{coro} Under the growth assumption
\begin{equation}
\rK(t,\rU)\le \phi(t)\rU\quad \hbox{for all $t\in[0,\rT]$ and $\rU\ge 0$},
\label{eq:boundedrate}
\end{equation}
for a function $\phi \in \rL^{p} (0,\rT),~p\in (1,\infty]$ the operator $\cG^{\ru_{0}}$ is a contraction on $\BB_{\rT}$ equipped with a kind of Bielecki (see \cite{Bielecki}) norm
\begin{equation}
\nn{\ru}_{\BB_{\rT}}\doteq \sup_{0\le t\le \rT}e^{-\gamma t}\n{\ru(t)}_{\cX},\quad \ru\in\BB_{\rT},
\label{eq.Bieleckinorm}
\end{equation}
provided $\gamma>\dfrac{p-1}{p}\n{\phi}_{\rL^{p}(0,\rT)}^{\frac{p}{p-1}}$ and $\phi\in\rL^{p}(0,\rT)$ for some $1<p<\infty$ or $\gamma >\n{\phi}_{\rL^{\infty}(0,\rT)}$ if $\n{\phi}_{\rL^{\infty}(0,\rT)}<\infty$. Therefore, the operator $\cG^{\ru_{0}}$ has a unique fixed point in both cases.
\label{coro:Gcontraction}
\end{coro}
{\sc Proof.} Applying the property \eqref{eq:boundedrate} to \eqref{eq:Kinequality0} we obtain
$$
\begin{array}{ll}
e^{-\gamma t}\n{\cG^{\ru_{0}}\ru(t)-\cG^{\ru_{0}}\rv(t)}_{\cX}&\hspace*{-.2cm}\disp \le \int^{t}_{0}e^{-\gamma t}\rK\big (s, \n{\ru(s)-\rv(s)}_{\cX}\big )ds\\ [.3cm]
&\hspace*{-.2cm}\disp \le \int^{t}_{0}e^{-\gamma s}\n{\ru(s)-\rv(s)}_{\cX}\phi(s)e^{-\gamma (t-s)}ds\\ [.3cm]
&\hspace*{-.2cm}\disp \le \nn{\ru-\rv}_{\BB_{\rT}}\int^{t}_{0}\phi(s)e^{-\gamma (t-s)}ds.
\end{array}
$$
For $\dfrac{1}{p}+\dfrac{1}{p'}=1$ Hölder's inequality leads to 
$$
e^{-\gamma t}\nl{\cG^{\ru_{0}}\ru(t)-\cG^{\ru_{0}}\rv(t)}_{\cX}\le \nn{\ru-\rv}_{\BB_{\rT}}\n{\phi}_{\rL^{p}(0,\rT)}\left (\int^{t}_{0}e^{-p'\gamma (t-s)}ds\right )^{\frac{1}{p'}}
$$
and
$$
\nnl{\cG^{\ru_{0}}\ru-\cG^{\ru_{0}}\rv}_{\BB_{\rT}}\le \dfrac{1}{(p'\gamma )^{\frac{1}{p'}} }\n{\phi}_{\rL^{p}(0,\rT)}\nn{\ru-\rv}_{\BB_\rT}
$$
follows.$\fin$
\par
\medskip
\noindent  {\sc Proof of Theorem \ref{teo:existence}} Adapting the reasoning of the first part of the proof of  Proposition \ref{prop:Gcontinuous} we obtain
\begin{equation}
\nl{\ru_{n+1}}_{\BB _{t}}\le \rU_{0}+\int^{t}_{0}\rK\big (s, \n{\ru_{n}}_{\BB_{\rT}}\big )ds,\quad n\ge 1.
\label{eq:norsucessiveapproximatio}
\end{equation}
Next we consider the global nonnegative solution $\rU$ on $[0,\rT]$ of the  integral equation 
\begin{equation}
\rU(t)=\rU_{0}+\int^{t}_{0}\rK\big (s, \rU(s)\big )ds,\quad t\in [0,\rT], 
\label{wq:functionu1}
\end{equation}
(see \eqref{eq:a2} in condition {\bf a2})) for which  one has
$$
\rU(t)-\nl{\ru_{n+1}}_{\BB_{t}}\ge \int^{t}_{0}\left (\rK\big (s,\rU(s)\big )-\rK\big (s,\nl{\ru_{n}}_{\BB_{s}}\big )\right )ds.
$$
Since $\rU(t)\ge \rU_{0}\doteq \n{\ru_{0}}_{\BB_{\rT}}$, the monotonicity of the function $\rK(s,\rU)$ on the second variable (see {\bf a1})) implies, by induction, the inequality
$$
\nl{\ru_{n}}_{\BB_{t}}\le \rU(t)\quad \hbox{for $t\in[0,\rT]$},
$$
where the function $\rU(t)$ is independent on $n$.  Thus, $\{\ru_{n}\}_{n\ge 0}$ is a bounded sequence in $\BB_{\rT}$. Analogously, for each $n\ge 0$
$$
\rR_{n}(t)\doteq \sup_{m\ge n}\nl{\ru_{m}-\ru_{n}}_{\BB_{t}}
$$
is a nonnegative, uniformly bounded, and nondecreasing function on $t\in [0,\rT]$. By construction, we may assume that for each $t\in [0,\rT]$ the real sequence $\{\rR_{n}(t)\}_{n\ge 0}$ is nonincreasing. It implies the existence of a nonnegative and nondecreasing function given by
$$
\rR(t)=\lim_{n\rightarrow \infty}\rR_{n}(t),\quad t\in [0,\rT].
$$
On the other hand, a similar reasoning as in the second part of the proof of  Proposition \ref{prop:Gcontinuous}, and using that the semigroup is of contractions, leads to
$$
\nl{\ru_{m}-\ru_{n}}_{\BB_{t}}\le \int^{t}_{0}\rK\big (s, \nl{\ru_{m-1}(s)-\ru_{n-1}(s)}_{\BB_{s}}\big )ds
$$
(see \eqref{eq:Kinequality}). Therefore, we obtain
$$
\rR(t)\le \rR_{n}(t)\le \int^{r}_{0}\rK\big (s,\rR_{n-1}(s)\big )ds,\quad t\in [0,\rT],
$$
whence the Lebesgue Convergence Theorem implies
$$
\rR(t)\le \int^{t}_{0}\rK\big (s,\rR(s)\big )ds,\quad t\in [0,\rT].
$$
So that we deduce $\rR(t)\equiv 0$ for $t\in [0,\rT]$ from the assumption {\bf a3}). Since
$$
\nl{\ru_{m}-\ru_{n}}_{\BB_{\rT}}\le \rR_{n}(\rT)
$$
we conclude
$$
\nl{\ru_{m}-\ru_{n}}_{\BB_{\rT}}\rightarrow 0\quad \hbox{as $m,n\rightarrow \infty$}.
$$
Therfore, we have proved the existence of a fixed point , $\ru$, of the operator $\cG^{\ru_{0}}$. 
Since 
$$
\nl{\ru}_{\BB_{\rT}}\le \nl{\ru-\ru_{n}}_{\BB_{\rT}}+ \nl{\ru_{n}}_{\BB_{\rT}}
$$
the estimate \eqref{eq:estimatebX} follows by repeating the reasonings to changing the global horizont $\rT$ by each $t\in [0,\rT]$.
\par
\noindent
Let us prove now that the operator $\cG^{\ru_{0}}$ only admits at most a fixed point in $\BB_{\rT}$. Indeed,  let  $\ru$ and $\widehat{\ru}$ be two eventual fixed points of $\cG^{\ru_{0}}$ and $\cG^{\widehat{\ru}_{0}}$ in $\BB_{\rT}$. Then
\begin{equation}
\nl{\ru-\widehat{\ru}}_{\BB_{t}}\le \n{\ru_{0}-\widehat{\ru}_{0}}_{\cX}+\int^{t}_{0}\rK\big (s, \n{\ru-\widehat{\ru}}_{\BB_{s}}\big )ds,\quad t\in [0,\rT]
\label{eq:poirfixesestomate}
\end{equation}
(see \eqref{eq:Gcontinuous}). Then from the property \eqref{eq:a3}, it follows $\nl{\ru-\widehat{\ru}}_{\BB_{t}}\equiv 0$ for all $t\in [0,\rT]$ whence $\ru= \widehat{\ru}_{0}$ in the space $\BB_{\rT}$, provided $\ru_{0}=\widehat{\ru}_{0}$.

\noindent Finally, to prove that $\ru$ can also be identified as the unique limit of the implicit Euler scheme associated with the
nonhomogeneous equation with the forcing term $g(t)=\rF(t,\ru(t))$ a.e. $t\in (0,\rT)$ it suffices to use the uniqueness of solution of Crandall-Liggett's Theorem and the fact that, as already shown, the mild solution is unique. $\fin$

\begin{rem}\rm We emphasize that the equality is required in \eqref{wq:functionu1} (see \eqref{eq:a2} in condition {\bf a2})) in the above proof. An inequality does not apply in the reasoning.$\fin$ 
\label{rem:a21}
\end{rem}

\section{The integral equation. Some uniqueness criteria revisited}
\label{sec:integralequation}

Certainly, as it follows from Theorem \ref{teo:existence}, the properties {\bf a2}) and {\bf a3})  are  key stones in the proofs. So, we dedicate this Section to study these main properties.  Theorems \ref{theo:Osgoodcriterion}, \ref{theo:Constantinextended} are revisited version of relative results in  \cite{Osgood1898,Nagumo,Levy,Constantin2010nagumo,Constantin2023uniqueness}.
\par
\noindent
In order to do it, in what follows we assume 
\begin{equation}
	\rK(t,\rU)=\phi(t)\vartheta(\rU),\quad t\in[0,\rT],~\rU\ge 0,
	\label{eq:Kchoice}
\end{equation}
where $\vartheta:~]0,\infty[\rightarrow ]0,\infty[$ is a continuous and increasing function verifying $\vartheta(0^{+})=0$ and $\phi$ is an integrable and non negative function on $[0,\rT]$. In this case, we repeat the proof of Theorem \ref{teo:existence} by using directly the function $\phi(t)\vartheta(\rU)$. Then we   replace the equation \eqref{wq:functionu1} by 
\begin{equation}
	\rU(t)=\rU_{0}+\int^{t}_{0}\phi(s)\vartheta \big (\rU(s)\big )ds, \quad 0\le s\le \rT,
	\label{eq:mainequation}
\end{equation}
in which we focus in order to prove whenever the assumption  {\bf a2}) and  {\bf a3}) are fulfilled.  
\begin{theo}[Osgood's criterion] When $\rU_{0}>0$ a positive function $\rU(t)$ satisfying 
\begin{equation}
\rU(t)\le \rU_{0}+\int^{t}_{0}\phi(s)\vartheta \big (\rU(s)\big )ds, \quad 0\le s\le \rT\le +\infty
\label{eq:maininequation}
\end{equation}
is given implicitly, in the whole interval $[0,\rT]$, by the property 
\begin{equation}
\int_{\rU_{0}}^{\rU(t)}\dfrac{ds}{\vartheta (s)}\le \int^{t}_{0}\phi(s)ds,\quad 0\le t\le \rT.
\label{eq:Leibnitzproperty}
\end{equation} 
Therefore the condition {\bf a2}) holds. Moreover, under the Osgood condition
\begin{equation}
\int_{0^{+}}\dfrac{d\rU}{\vartheta(\rU)}=\infty ,
\label{eq:Osgoodcondition}
\end{equation}
the unique nonnegative solution $\rU(t)$ of 
$$
\rU(t)\le \int^{t}_{0}\phi (s)\vartheta \big (\rU(s)\big )ds, \quad 0\le s\le \rT
$$ 
is the null function, $i.e.$ we have {\bf a3}).
\label{theo:Osgoodcriterion}
\end{theo}
{\sc Proof.} As it is well known, integral equation as \eqref{eq:mainequation} is equivalent to the Cauchy problem
\begin{equation}
\left \{
\begin{array}{l}
	\rU'(t)=\phi(t)\vartheta \big (\rU(t)\big ),\quad \\
	\rU(0)=u_{0}>0
\end{array}
\right .
\label{eq:Cuachyproblem}
\end{equation}
whose positive and continuous global solution is represented  by
$$
\int_{\rU_{0}}^{\rU(t)}\dfrac{ds}{\vartheta (s)}=\int^{t}_{0}\phi(s)ds,\quad 0\le t\le \rT.
$$
When $\rU(t)$ solves the inequality \eqref{eq:maininequation} we require a sharp refinement because an equivalence as  \eqref{eq:mainequation} and \eqref{eq:Cuachyproblem} does not hold in general. So, we introduce the positive and non decreasing  function $\disp \rV(t)=\max_{0\le \tau\le t} \rU(\tau)=\rU(\tau_{t})$, for some  $\tau_{t}\in [0,t]$. Next, we define
$$
\widehat{\rV}(t)= \rU_{0}+\int_{0}^{t}\phi(s)\vartheta \big (\rV(s)\big )ds > 0,\quad t>0
$$
that satisfies $\widehat{\rV}(0)=u_{0}$ as well as
$$
\rV(t)=\rU(\tau_{t})\le \rU_{0}+\int_{0}^{\tau_{t}}\phi(s)\vartheta \big (\rU(s)\big )\le \rU_{0}+\int_{0}^{t}\phi(s)\vartheta \big (\rV(s)\big )=\widehat{\rV}(t) 
$$
and
$$
\left \{
\begin{array}{l}
\widehat{\rV}'(t)=\phi(t)\vartheta\big (\rV(t)\big )\le \phi(t)\vartheta\big (\widehat{\rV}(t)\big ),\\ [.15cm]
\widehat{\rV}(0)=\rU_{0}>0
\end{array}
\right .
$$
close to \eqref{eq:Cuachyproblem}. Hence, a kind of Leibnitz inquality
$$
\int^{\rU(t)}_{\rU_{0}}\dfrac{dr}{\vartheta(r)}\le \int^{\widehat{\rV}(t)}_{\rU_{0}}\dfrac{dr}{\vartheta(r)}=\int ^{t}_{0}\dfrac{\widehat{\rV}'(t)dt}{ \vartheta\big (\widehat{\rV}(t)\big )}\le  \int^{t_{1}}_{0}\phi(s)ds<+\infty
$$
holds, whence \eqref{eq:Leibnitzproperty} follows.
\par
\noindent On the other hand, when $\rU_{0}=0$, if we suppose $\rU(t)>0$ in some interval $t\in ]0 ,t_{1}]\subset [0,\rT]$. the above reasoning shows that one satisfies $\widehat{\rV}(0)=0,~0<\rV(t)\le \widehat{\rV}(t)$ and
$ \widehat{\rV}'(t)\le \phi(t)\vartheta\big (\widehat{\rV}(t)\big )$. Then  we deduce $\widehat{\rV}(t)>0$ in $t\in ]0 ,t_{1}]$ and 
$$
\int^{\widehat{\rV}(t_{1})}_{0}\dfrac{dr}{\vartheta(r)}=\int ^{t_{1}}_{0}\dfrac{\widehat{\rV}'(t)dt}{ \vartheta\big (\widehat{\rV}(t)\big )}\le  \int^{t_{1}}_{0}\phi(s)ds<+\infty
$$
contrary to the condition \eqref{eq:Osgoodcondition}.$\fin$
\begin{rem}\rm 
Since $\vartheta$ is a continuous and increasing function the relative equation \eqref{eq:a2} becomes
\begin{equation}
 \int^{\rU(t)}_{\rU_{0}}\dfrac{ds}{\vartheta (s)}=\int^{t}_{0}\phi(s)ds,\quad t\ge 0.
 \label{eq:a2separable}
\end{equation}
\fineqnum
\label{rem:a22}
\end{rem}
Let us introduce the increasing function
\begin{equation}
\Psi_{\rU_{0}}(\rU)\doteq \int^{\rU}_{\rU_{0}}\dfrac{ds}{\vartheta (s)},\quad \rU\ge \rU_{0},
\label{eq:Psifunction}
\end{equation} 
provided $\rU_{0}>0$. Then we have the estimate \eqref{eq:maininequation}  by  the version
\begin{equation}
	\rU(t)=\Psi ^{-1}_{\rU_{0}}\left (\int^{t}_{0}\vartheta(s)ds\right ),\quad 0\le t\le \rT,
	\label{eq:solutionintegralequation}
\end{equation} 
of \eqref{eq:Leibnitzproperty}, where $\rU(0)=\Psi ^{-1}(0)=\rU_{0}>0$. This function is defined by horizont $\rT$ such that
\begin{equation}
\int ^{+\infty}_{\rU_{0}}\dfrac{ds}{\vartheta (s)}\ge\int ^{\rT}_{0}\phi(s)ds.
\label{eq:horionU}
\end{equation}
\begin{rem}\rm We emphasize that in the Osgood assumption \eqref{eq:Osgoodcondition} only the behaviour of the function $\vartheta$ near the origin is involved.$\fin$
\label{rem:behaviornearoriginOsgood}
\end{rem}  
\begin{coro}Under \eqref{eq:Kchoice}  implies
\begin{equation}
\int^{\nl{\ru-\widehat{\ru}}_{\BB_{t}}}_{\n{\ru_{0}-\widehat{\ru}_{0}}_{\cX}}\dfrac{ds}{\vartheta(s)}\le \int^{t}_{0}\phi(s)ds,\quad 0\le t\le \rT
\label{eq:generalestimate}
\end{equation}
provided $\n{\ru_{0}-\widehat{\ru}_{0}}_{\cX}>0$.
Moreover, under the Osgood condition \eqref{eq:Osgoodcondition} one has the property
$$
\n{\ru_{0}-\widehat{\ru}_{0}}_{\cX}=0\quad \Rightarrow \quad  \nl{\ru-\widehat{\ru}}_{\BB_{t}}=0.
$$
\fineq
\label{coro:uniquenessestimate}
\end{coro}
\noindent {\sc Proof.} Since $\disp \int ^{\nl{\ru-\widehat{\ru}}_{\BB_{t}}}_{0^+}\dfrac{ds}{\vartheta (s)}=+\infty$  the substitution  $\n{\ru_{0}-\widehat{\ru}_{0}}_{\cX}=0$ in \eqref{eq:generalestimate} only holds if the upper limit is $\nl{\ru-\widehat{\ru}}_{\BB_{t}}=0.\fin$
\begin{rem}\em
We note that no convex function $\vartheta(\rU)$ satisfies \eqref{eq:Osgoodcondition}. Certainly, if \eqref{eq:Osgoodcondition} holds then the function $\vartheta$ is not integrable near 0. For instance, the functions satisfying 
\begin{equation}
\dfrac{\vartheta(\rU)}{\rU}\le  \ln \dfrac{1}{\rU},\quad  \ln\dfrac{1}{\rU} \ln \stackrel{n}{\cdots}\ln\dfrac{1}{\rU},\quad n\ge 0,\quad \hbox{near $\rU=0$}
\label{eq:mainexampleOsgood}
\end{equation}
provide examples for which \eqref{eq:Osgoodcondition} holds near the origin. The conditions \eqref{eq:Osgoodcondition} and \eqref{eq:mainexampleOsgood} coincide with the classical Osgood's criterion (see \cite{Osgood1898}).
\par
\noindent On the other hand, we also note that the examples given by \eqref{eq:mainexampleOsgood} verify
$$
\vartheta (\rU)|\ln \rU|\le \rU |\ln \rU|\ln \dfrac{1}{\rU}\ln \stackrel{n}{\cdots}\ln\dfrac{1}{\rU}=\rU\big (\ln \rU\big )^{n+2},\quad n\ge 0,\quad \hbox{near $\rU=0$},
$$
thus they satisfy the Dini condition
\begin{equation}
	\lim _{u \searrow 0}\vartheta (\rU)|\ln \rU|=0.
	\label{eq:Dinicondition}
\end{equation}
We emphasize that, as in \eqref{eq:Osgoodcondition} only the behaviour of $\vartheta$ near the origin is involved  in the Dini condition.$\fin$
\end{rem}
\par
Next we give some simple power like cases $\vartheta _{m}(\rU)=\rU^{m},~m>0$, for which the assumptions {\bf a2}) and {\bf a3}) hold. Certainly, these criteria must satisfy \eqref{eq:Osgoodcondition}, thus
$$
\int_{0^+}\dfrac{d\rU}{\rU^{m}}=+\infty.
$$
We note that the assumption \eqref{eq:Osgoodcondition} holds if and only if $m\ge 1$.
\par
\noindent Certainly, the most famous uniqueness criteria is related to Gronwall`s Lemma whenever one considers the Lipschitz case
\begin{equation}
\vartheta _{1}(\rU)=\rU,\quad \rU\ge 0.
\label{eq:Lipschitzcase}
\end{equation}
\begin{coro}[Gronwall's inequality] Let $\rU(t)$ verifying
$$
0\le \rU(t)\le \rU_{0}+\int^{t}_{0}\phi(s)\rU(s)ds<+\infty,\quad 0\le t,
$$
provided $\rU_{0}\ge 0$ and  $\phi\in \rL^{1}_{\hbox{\tiny loc}}(0,+\infty),~\phi \ge 0$, then
\begin{equation}
	0\le \rU(t)\le \rU_{0}\exp \left (\int^{t}_{0}\phi (s)ds \right ),\quad 0\le t.
	\label{eq:estiGronwallabstrac}
\end{equation}
In particular, $\rU_{0}=0$ implies $\rU(t)\equiv 0.\fin$ 
\label{coro:Gronwallinequality}
\end{coro}
\begin{rem}\rm It follows from the representation \eqref{eq:estiGronwallabstrac} in which
\eqref{eq:Psifunction} is given by 
$$	
\Psi_{\rU_{0}}(\rU)\doteq \int^{\rU}_{\rU_{0}}\dfrac{ds}{s}=\ln \dfrac{\rU}{\rU_{0}},\quad \rU\ge \rU_{0}>0.
$$
Here the maximal horizon $\rT_{\infty}=+\infty $ is independent on $\rU_{0}$ (see \eqref{eq:horionU}) and \eqref{eq:generalestimate} becomes  
\begin{equation}
\nl{\ru-\widehat{\ru}}_{\BB_{t}}\le \n{\ru_{0}-\widehat{\ru}_{0}}_{\cX}\exp \left (\int^{t}_{0}\phi(s)ds\right ),\quad 0\le t\le \rT<\infty. 
\label{eq:abstracestimate1}
\end{equation}
Clearly,  
$$
\n{\ru_{0}-\widehat{\ru}_{0}}_{\cX}=0\quad \Rightarrow \quad  \nl{\ru-\widehat{\ru}}_{\BB_{t}}=0.
$$
\fineq
\label{rem:Gronwall}
\end{rem}
The other power like criterion is governed by  $\vartheta _{m}(\rU)=\rU^{m},~m>1$.
\begin{coro}[Power like criterion] Let $\rU(t)$ verifying
$$
0\le \rU(t)=\rU_{0}+\int^{t}_{0}\phi(s)\big (\rU(s)\big )^{m}ds<+\infty,\quad 0\le t< \rT_{\infty}<\infty ,
$$
provided $u_{0}\ge 0,~m>1$ and  $\phi\in \rL^{1}(0,\rT),~\phi \ge 0$, then
\begin{equation}
	\big (\rU(t)\big )^{1-m}+(m-1)\int^{t}_{0}\phi (s)ds=\rU_{0}^{1-m},\quad 0\le t < \rT_{\infty}<\infty ,
	\label{eq:estipowerabstrac}
\end{equation}
provided $\rU_{0}>0$. Here the maximal horizon $\rT_{\infty}<\infty $ is dependent on $\rU_{0}$ given by 
\begin{equation}
	\rU_{0}^{1-m}=(m-1)\int^{\rT_{\infty}(\rU_{0})}_{0}\phi (s)ds
	\label{eq:hirizonm}
\end{equation}
(see \eqref{eq:horionU}). Moreover, $\rU_{0}=0$ implies $\rU(t)\equiv 0.\fin$
\label{coro:powerlike}
\end{coro}
\begin{rem}\rm Now the representation \eqref{eq:estipowerabstrac} is given by 
the version of  \eqref{eq:Psifunction}
$$	
\Psi_{\rU_{0}}(\rU)\doteq \int^{\rU}_{\rU_{0}}\dfrac{ds}{s^{m}}=\dfrac{1}{m-1}\left( \dfrac{1}{\rU_{0}^{m-1}}-\dfrac{1}{\rU^{m-1}}\right ),\quad \rU\ge \rU_{0}>0.
$$
Then  \eqref{eq:generalestimate} becomes
\begin{equation}
\n{\ru_{0}-\widehat{\ru}_{0}}_{\cX}^{1-m}\le \nl{\ru-\widehat{\ru}}_{\BB_{t}}^{1-m}+(m-1)\int^{t}_{0}\phi (s)ds,\quad 0\le t\le 
\rT_{\infty}(\n{\ru_{0}}_{\cX})<+\infty, 
\label{eq:examplepower}
\end{equation}
provided $\n{\ru_{0}-\widehat{\ru}_{0}}_{\cX}>0$. Once more, 
$$
\n{\ru_{0}-\widehat{\ru}_{0}}_{\cX}=0\quad \Rightarrow \quad  \nl{\ru-\widehat{\ru}}_{\BB_{t}}=0.
$$
Indeed, the substitution  $\n{\ru_{0}-\widehat{\ru}_{0}}_{\cX}=0$ in \eqref{eq:examplepower} implies
$$
+\infty\le \nl{\ru-\widehat{\ru}}_{\BB_{t}}^{1-m}+(m-1)\int^{t}_{0}\phi (s)ds,\quad 0\le t\le \rT_{\infty}(\n{\ru_{0}}_{\cX})<+\infty, 
$$
whence 
$$
\nl{\ru-\widehat{\ru}}_{\BB_{t}}^{1-m}=+\infty \quad \Rightarrow \quad \nl{\ru-\widehat{\ru}}_{\BB_{t}}=0.
$$
\fineqnum
\label{rem:powerlike}
\end{rem}
\begin{exam}\rm For the choice in Example \ref{exam:beta} the inequality \eqref{eq:maininequation} becomes
\begin{equation}
\rU(t)=\rU_{0}+\bS_{0}\int^{t}_{0}\vartheta \big (\rU(s)\big )ds, \quad 0\le s\le \rT\le +\infty
\label{eq:maininequationbeta}.
\end{equation}
where $\bS_{0}$ is a positive constant and $\vartheta (\rU)=\rU\ln \rU,~\rU\ge 0$ (see \eqref{eq:modulocoalbedo}). In view of 	\eqref{eq:solutionintegralequation} the version of  \eqref{eq:Psifunction} is 
\begin{equation}
\int^{\nl{\ru-\widehat{\ru}}_{\BB_{t}}}_{\n{\ru_{0}-\widehat{\ru}_{0}}_{\cX}}\dfrac{ds}{\vartheta(s)}\le \bS_{0}t,\quad 0\le t ,
\label{eq:betaestimate}
\end{equation}
provided $\n{\ru_{0}-\widehat{\ru}_{0}}_{\cX}>0$. In particular
\begin{equation}
\int^{\nl{\ru-\widehat{\ru}}_{\BB_{t}}}_{\nl{\ru_{0}-\widehat{\ru}_{0}}_{\cX}}
\dfrac{d s}{s\ln s }\le \bS_{0}t, \quad 0\le t 
\label{eq:coalbedoestimate}
\end{equation}
provided $0<\nl{\ru_{0}-\widehat{\ru}_{0}}_{\cX}\le \nl{\ru-\widehat{\ru}}_{\BB_{t}}$. As the Osgood condition \eqref{eq:Osgoodcondition} holds, we have 
$$
\n{\ru_{0}-\widehat{\ru}_{0}}_{\cX}=0\quad \Rightarrow \quad  \nl{\ru-\widehat{\ru}}_{\BB_{t}}=0.
$$
Indeed, as in Coriollary \ref{coro:uniquenessestimate}, since $\disp \int ^{\nl{\ru-\widehat{\ru}}_{\BB_{t}}}_{0^+}\dfrac{ds}{s\ln s}=+\infty$  the substitution  $\n{\ru_{0}-\widehat{\ru}_{0}}_{\cX}=0$ in \eqref{eq:coalbedoestimate} only holds if the upper limit is $\nl{\ru-\widehat{\ru}}_{\BB_{t}}=0$.
\par \noindent
More precisely, in this case, one has
\begin{equation}
\int^{\rU}_{\rU_{0}}\dfrac{ds}{\vartheta (s)}=\ln \left (\dfrac{\ln \rU}{\ln \rU_{0}}\right )\quad \rU\ge \rU_{0},
\label{eq:Psifunctioncoalbedo}
\end{equation} 
provided $\rU_{0}>1$. Here the maximal horizon $\rT_{\infty}=+\infty $ is independent on $\rU_{0}$ (see \eqref{eq:horionU}). Then  \eqref{eq:coalbedoestimate} becomes
\begin{equation}
\nl{\ru-\widehat{\ru}}_{\BB_{t}}\le \nl{\ru_{0}-\widehat{\ru}_{0}}_{\cX}^{e^{\rS_{0}t}},\quad 0\le t.
\label{eq:coalbedoestimate1}
\end{equation}
\fineqnum
\label{exam:beta2}
\end{exam}

Related to the study of the property {\bf a3}) we may replace the inequality \eqref{eq:Kchoice} by Lévy's extension (see \cite[Theoréme 10.I]{Levy}) of the well-known Nagumo`s criterion (see \cite{Nagumo}). Here we incorporate an improved generalization based on \cite{Constantin2010nagumo}. This corresponds with the case 
\begin{equation}
\rK(t,\rU)\le\dfrac{\psi'(t)}{\psi(t)}\vartheta(\rU),
\label{eq:Levyextension}
\end{equation} 
where $\psi :]0,\infty[\rightarrow ]0,\infty[$ is a differentiable function satisfying
$$
\psi(0^{+})=0,~\psi '(0^{+})>0\quad \hbox{and} \quad \psi '(t)>0
$$
and $\vartheta:~]0,\infty[\rightarrow ]0,\infty[$ is a continuous and increasing function verifying $\vartheta(0^{+})=0$ and the Nagumo condition
\begin{equation}
\int^{r}_{0}\dfrac{\vartheta(s)}{s}ds\le r,\quad r\ge 0.
\label{eq:Nagumolarge}
\end{equation}
Then we may prove {\bf a3}). The reasonings involve the functional space
$$
\cZ=\left \{\rU\in\cC\big ([0,\rT]:[0,\infty[\big )\hbox{ such that }\lim_{t\searrow 0}\dfrac{\rU(t)}{\psi(t)}=0\right \}
$$
endowed with the norm
$$
\n{\rU}\doteq \sup_{0\le t\le \rT}\dfrac{\rU(t)}{\psi(t)}.
$$
In fact, under 	\eqref{eq:Levyextension}, \eqref{eq:Nagumolarge} and \eqref{eq:NagumoZ} we claim that a generalized version of Nagumo's criterion holds, $i.e.$ the null function is the unique nonnegative solution of \eqref{eq:a3}  in the space $\cZ$ verifying \eqref{eq:a3}. Indeed, assume that there exists $u\in\cZ$ such that 
$$
0\neq \n{\rU}=\sup_{0\le t\le \rT}\dfrac{\rU(t)}{\psi(t)}=\dfrac{\rU(t^{*})}{\psi(t^{*})}. 
$$
Since $\rU\in\cZ\setminus \{0\}$ and $\rU(0)=0$ by construction we have $\rU(t)\not\equiv \n{\rU}$ if $t\in [0,t^{*}]$, for some $t^{*}$ that without loss of generality we may assume small. From \eqref{eq:NagumoC} we deduce the inequality
\begin{equation}
\begin{array}{ll}
	\rU(t^{*})&\hspace*{-.2cm}\disp \le \int_{0}^{t^{*}}\rK \big (s,\rU(s)\big )ds\le  \int_{0}^{t^{*}}\dfrac{\psi'(s)}{\psi(s)}\vartheta(\rU(s))ds<
	\int_{0}^{t^{*}}\dfrac{\psi'(s)}{\psi(s)}\vartheta\big (\psi (s)\n{\rU}\big )ds\\ [.35cm]
	&\hspace*{-.2cm}\disp =
	\int^{\infty}_{\tau^{*}}\vartheta\big (e^{-\tau}\n{\rU}\big )d\tau= \int^{\psi(\tau^{*})\n{\rU}}_{0}\dfrac{\vartheta(r)}{r}dr\le \psi (t^{*})\n{\rU},
\end{array}
\label{eq:Nagumoreasonings}
\end{equation}
where first we have used the change of variable $\tau =-\ln \psi(s)$ and then $r=\n{\rU}e^{-\tau}$. Therefore, we have obtained a contradiction
$$
\n{\rU}<\n{\rU}
$$
and the claim holds.
\begin{rem}\rm It follows that under the assumption
\begin{equation}
\lim_{t\searrow 0}\dfrac{\rK(t,\rU)}{\psi'(t)}=0\quad \hbox{uniformly to small values of $\rU$},
\label{eq:NagumoZ}
\end{equation}
any nonnegative continuous function, $\rU$, verifying the  integral inequation 
$$
\rU(t)\le \int_{0}^{t}\rK \big (s,\rU(s)\big )ds ,\quad t\in [0,\rT]
$$
(see \eqref{eq:a3}) belongs to $\cZ$. Indeed, if $\rU$ solves  \eqref{eq:a3} one has
$$
\dfrac{\rU(t)}{\psi(t)}\le \dfrac{\disp \int^{t}_{0}\rK\big (s,\rU(s)\big )ds}{\psi(t)}
\le \dfrac{\disp\int^{t}_{0}\rK\big (s,\rM_{t_{0}}\big )ds}{\psi(t)},\quad 0\le t\le t_{0}
$$
where $\disp \rM_{t_{0}}=\sup_{0\le t\le t_{0}}\rU(t)$. Hence  by l'Hopital's rule, one obtains
$$
\lim_{t\searrow 0}\dfrac{\rU(t)}{\psi(t)}\le \lim_{t\searrow 0}\dfrac{\rK (t,\rU)}{\psi'(t)}\quad \hbox{uniformly in small values of $\rU$}.
$$
\fineqnum
\end{rem}
\begin{rem}\rm In proving \eqref{eq:Nagumoreasonings} it is enough  to use  
\begin{equation}
\int^{r}_{0}\dfrac{\vartheta(s)}{s}ds\le r,\quad 0\le r\le r^{*}
\label{eq:NagumoC}
\end{equation}
where $r^{*}$ can be a small positive value. In this case, it is enough to choose  $t^{*}$ appropriately small.$\fin$
\label{rem:behaviornearoriginNagumo}
\end{rem}
We note that $\vartheta (s)=s $ is the simplest case satisfying \eqref{eq:NagumoC}. It coincides with Nagumo's classical criterion. In \cite{Constantin2010nagumo}, it was given another example verifying \eqref{eq:NagumoC}.
On the other hand, we also note the inequality
$$
\vartheta\left (\dfrac{r}{2}\right )\ln 2=\vartheta \left (\dfrac{r}{2}\right )\int^{r}_{\frac{r}{2}}\dfrac{1}{s}ds\le \int^{r}_{\frac{r}{2}}\dfrac{\vartheta (s)}{s}ds\le \int^{r}_{0}\dfrac{\vartheta (s)}{s}ds,
$$
whence  \eqref{eq:NagumoC} leads to 
$$
\vartheta \left (\dfrac{r}{2}\right )\left |\ln \dfrac{r}{2}\right | \le \dfrac{1}{\ln 2}r,
$$
thus the integral condition \eqref{eq:NagumoC} also implies the Dini condition \eqref{eq:Dinicondition}.
\par
Next, we deal with suitable combinations, non-necessarily convex,  of Osgood and Nagumo uniqueness criteria. It extends \cite[Theorem 2.1]{Constantin2023uniqueness}. More precisely, here we consider
\begin{equation}
\rK(t,\rU)\le \mu \vartheta(t)\vartheta_{{\rm O}}(\rU)+\lambda \dfrac{\psi'(t)}{\psi(t)} \vartheta_{{\rm N}}(\rU),\quad t\in [0,\rT],~\rU\ge 0,~0\le \mu, 0\le \lambda \le 1
\label{eq:ConstantinG}
\end{equation}
where $\vartheta$ is integrable on $[0,\rT] $ as well as $\vartheta_{{\rm O}},~ \vartheta_{{\rm N}}:[0,\infty[\rightarrow [0,\infty[$ are continuos and nondecreasing functions such that $\vartheta _{{\rm O}}(0)=\vartheta_{{\rm N}}(0)=0$ and \eqref{eq:Osgoodcondition} and \eqref{eq:NagumoC} respectively hold. Moreover, we assume that $\psi :]0,\infty[\rightarrow ]0,\infty[$ is a differentiable function satisfying
$$
\psi(0^{+})=0,~\psi '(0^{+})>0\quad \hbox{and} \quad \psi '(t)>0.
$$
We note $\psi (t)=t$ is the simplest case of the function $\psi$. Really, it is sufficient that $\psi $ let be defined in a bounded interval $]0,\rL[,~\rL$ small. Clearly, the inequality \eqref{eq:ConstantinG} includes the Lipschitz case \eqref{eq:Lipschitzcase}.
\begin{theo} Let us assume \eqref{eq:ConstantinG} as before. Let $u$ be a nonnegative continuous  in the space $\cZ$ veryfying the  inequality 
$$
\rU(t)\le \int_{0}^{t}\rK \big (s,\rU(s)\big )ds,\quad t\in[0,\rT].
$$
Then if $\dfrac{\vartheta}{\psi}$ is integrable on $[0,\rT] $ the function $\rU(t)$ is the null function.
\label{theo:Constantinextended}
\end{theo}
{\sc Proof.} We emphasize that the cases $(\mu,\lambda)=(1,0)$ and $(\mu,\lambda)=(0,1)$ are proved in Osgood's criterion and Nagumo`s criterion  (see (Theorem  \ref{theo:Osgoodcriterion} and the above arguments) respectively. So for $\mu>0$ and $0<\lambda <1$ we form
$$
\dfrac{\rU(t)}{\psi(t)}\le \dfrac{\mu }{\psi(t)}\int^{t}_{0}\vartheta(s)\vartheta_{{\rm O}}\big (\rU(s)\big )ds
+\dfrac{\lambda}{\psi(t)}\int^{t}_{0}\dfrac{\psi'(s)}{\psi(s)} \vartheta_{{\rm N}}\big (\rU(s)\big )ds,\quad t\in [0,\rT].
$$
In particular, the function $\hU(t)=\dfrac{\rU(t)}{\psi(t)}$ satisfies
$$
\begin{array}{ll}
\hU(t)&\hspace*{-.2cm}\disp \le \dfrac{\mu}{\psi(t)}\int^{t}_{0}\vartheta(s)\vartheta_{{\rm O}}\big (\psi (s)\hU(s)\big )ds
+\dfrac{\lambda}{\psi(t)}\int^{t}_{0}\dfrac{\psi'(s)}{\psi(s)}\vartheta_{{\rm N}}\big (\psi (s)\hU(s)\big )ds\\ [.3cm]
&\hspace*{-.2cm}\disp \le \mu \int^{t}_{0}\dfrac{\vartheta(s)}{\psi(s)}\vartheta_{{\rm O}}\big (\psi(\rT)\hU(s)\big )ds +\dfrac{\lambda}{\psi(t)}\int^{t}_{0}\dfrac{\psi'(s)}{\psi(s)}\vartheta_{{\rm N}}\big (\psi (s)\hU(s)\big )ds,\quad t\in [0,\rT].
\end{array}
$$
We note that by
$$
\lim_{t\searrow 0}\dfrac{\rU(t)}{\psi(t)}=\lim_{t\searrow 0}\dfrac{\rK(t,\rU)}{\psi'(t)}=0,
$$
the function $\hU$ is continuous in $[0,\rT]$.
Next, introducing the nondecreasing function $\disp \rV(t)\doteq\max_{0\le \tau\le t} \hU(\tau)=\hU(\tau_{t})$, for some  $\tau_{t}\in [0,t]$, we get
$$
\hU(\tau)\le \mu\int^{\tau}_{0}\dfrac{\vartheta(s)}{\psi(s)}\vartheta_{{\rm O}}\big (\psi(\rT)\rV(s)\big )ds +\dfrac{\lambda }{\psi(\tau)}\int^{\tau}_{0}\dfrac{\psi'(s)}{\psi(s)}\vartheta_{{\rm N}}\big (\psi (s)\rV(\tau)\big )ds,\quad 0<\tau<t\le \rT.
$$
As in the reasoning of Nagumo's criterion (see \eqref{eq:Nagumoreasonings})
$$
\int_{0}^{\tau}\dfrac{\psi'(s)}{\psi(s)}\vartheta_{{\rm N}}\big (\psi (s)\rV(\tau)\big )ds=
\int^{\infty}_{-\ln \psi(\tau)}\vartheta_{{\rm N}}\big (e^{-\widehat{\tau}}\rV(\tau)\big )d\widehat{\tau}=\int^{\psi (\tau)\rV(\tau)}_{0}\dfrac{\vartheta_{{\rm N}}(r)}{r}dr\le \psi (\tau)\rV(\tau),
$$
where first we have used the change of variable $\widehat{\tau} =-\ln \psi(s)$ and then $r=\rV(\tau)e^{-\widehat{\tau}}$.
So that
$$
\begin{array}{ll}
\hU(\tau) 
&\hspace*{-.2cm}\disp \le \mu \int^{\tau}_{0}\dfrac{\vartheta(s)}{\psi(s)}\vartheta_{{\rm O}}\big (\psi(\rT)\rV(s)\big )ds +\lambda \rV(\tau)\\ [.3cm]
&\hspace*{-.2cm}\disp \le \mu \int^{t}_{0}\dfrac{\vartheta(s)}{\psi(s)}\vartheta_{{\rm O}}\big (\psi(\rT)\rV(s)\big )ds +\lambda \rV(t),\quad 0<\tau<t\le \rT,
\end{array}
$$
and by construction, we obtain
$$
\rV(t)\le \mu \int^{t}_{0}\dfrac{\vartheta(s)}{\psi(s)}\vartheta_{{\rm O}}\big (\psi(\rT)\rV(s)\big )ds +\lambda \rV(t),\quad 0<t\le \rT.
$$
Thus
$$
0\le \psi(\rT)\rV(t)\le \dfrac{\mu}{1-\lambda}\int^{t}_{0}\psi(\rT)\dfrac{\vartheta(s)}{\psi(s)}\vartheta_{{\rm O}}\big (\psi(\rT)\rV(s)\big )ds,\quad 0<t\le \rT.
$$
We need a simple change of notation to use the reasoning of Theorem \ref{theo:Osgoodcriterion}. In particular, since the function $\widehat{\vartheta}(t)= \psi(\rT)\dfrac{\vartheta(t)}{\psi(t)}$ is integrable on $[0,\rT]$ and $v (0)=0$, by the Osgood's condition \eqref{eq:Osgoodcondition} one concludes
$$
\psi(\rT)\rV(t)\equiv 0\quad \hbox{in $[0,\rT]$}\quad \Rightarrow \quad 
\rU (t)\equiv 0\quad \hbox{in $[0,\rT]$},
$$
as in Theorem \ref{theo:Osgoodcriterion}.$\fin$
\begin{rem}\rm  Once more we emphasize that in assumption \eqref{eq:ConstantinG} only the behaviour of the function $\vartheta$ near the origin is involved (see Remarks \ref{rem:behaviornearoriginOsgood} and \ref{rem:behaviornearoriginNagumo}).$\fin$ 
\label{rem:smallnature}
\end{rem}
\par
\medskip
So that we may extend the Theorem \ref{theo:Osgoodcriterion} when
\begin{equation}
\vartheta(t)\le\dfrac{\psi'(t)}{\psi(t)}
\label{eq:extendingphi}
\end{equation}
where $\psi :]0,\infty[\rightarrow ]0,\infty[$ is a differentiable function satisfying
$$
\psi(0^{+})=0,~\psi '(0^{+})>0\quad \hbox{and} \quad \psi '(t)>0
$$
\begin{coro}When $\rU_{0}>0$ any positive function $\rU(t)$ satisfying 
\begin{equation}
\rU(t)\le \rU_{0}+\int^{t}_{0}\dfrac{\psi'(s)}{\psi(s)}\vartheta \big (\rU(s)\big )ds, \quad 0\le s\le \rT\le +\infty
\label{eq:maininequationNagumo}
\end{equation}
is given implicitly in the whole interval $[0,\rT]$ by the property 
\begin{equation}
\int_{\rU_{0}}^{\rU(t)}\dfrac{ds}{\vartheta (s)}\le \psi(t),\quad 0\le t\le \rT.
\label{eq:LeibnitzpropertyNagumo}
\end{equation} 
Therefore, the condition {\bf a2}) holds. Moreover, under the assumptions of Theorem \ref{theo:Constantinextended}, the property {\bf a3}) holds.$\fin$
\label{coro:Osgoodcriterionextended}
\end{coro}
\begin{rem}\rm The Corollaries \ref{coro:Gronwallinequality} and \ref{coro:powerlike} and the inequalities \eqref{eq:abstracestimate1} and \eqref{eq:examplepower} are immediately adapted to the
case  \eqref{eq:extendingphi}.$\fin$
\end{rem}

\begin{tabular}{ll}
Gregorio Díaz & Jesús Ildefonso  Díaz \\
& Instituto Matemático Interdisciplinar (IMI)\\
Dpto. Análisis Matemático & Dpto. Análisis Matemático \\
y Matem\'atica Aplicada & y Matem\'atica Aplicada \\
U. Complutense de Madrid & U. Complutense de Madrid \\
28040 Madrid, Spain & 28040 Madrid, Spain\\
{\tt gdiaz@ucm.es} & {\tt Jidiaz@ucm.es}
\end{tabular}


\begin{thebibliography}{99}

\bibitem{Alekseev} Alekseev, V.M.: An estimate for the perturbations of the
solutions of ordinary differential equations (Russian), {\em Vestnik Moskov Univ}.
Ser. I Mat. Meh., {\bf 2} (1961), 2836.

\bibitem{Amann1972} Amann, H.: Fixed point equations and nonlinear eigenvalue problems in ordered Banach spaces, {\em SIAM Review}, {\bf 18} (1976), 620--709.

\bibitem{Barbu-libro} Barbu, V.: {\em Nonlinear Differential Equations of Monotone Types in Banach Spaces}, Springer, New York, 2010.

\bibitem{BarbuBocsan2002}Barbu, D. and Boc\c{s}an, G.: Approximation to mild solutions of stochastic semilinear equations with non-Lipschitz coefficients,
{\em Czechoslovak Mathematical Journal}, {\bf 127} (2002), 87--95.

\bibitem{Benilantesis} B\'{e}nilan, Ph.: {\em Equat\'{\i}ons D'evolut\'{\i}on Dans un Espace de Banach Quelconque et Applications}, Th\'{e}se, Orsay, 1972.

\bibitem{BCP}Benilan, Ph.,  Crandall. M.G. and Pazy. A.: {\em Nonlinear Evolution Equations in Banach space}, Available online, unpublished book.


\bibitem{Bensid-D3} Bensid, S.  and D\'{\i}az, J.I.: Discontinuous Bifurcation
in Emissivity of Monotone Solutions to a Budyko-Type Quasilinear Climate
Equation. Submitted.

\bibitem{Bielecki}Bielecki, A.:  Une remarque sur la méthode de Banach-Cacciopoli-Tikhonov dans la théorie  des équations differetielles ordinaires, {\em Bull. Acad. Polon. Sci.}, Cl. III, {\bf 4} (1956), 261-264.

\bibitem{Brezis} Brézis, H.: {\em Operateurs maximaux monotones et semi-groupes de contractions dans les espaces de Hilbert}. Amsterdam. North-Holland,
1973.

\bibitem{CasalDDVegas} Casal, A.C., D\'{\i}az, G., D\'{\i}az, J.I.  and 
Vegas, J.M.: Controlled boundary explosions: dynamics after blow-up for some
semilinear problems with global controls, {\em Discrete Contin. Dyn. Syst},
doi:10.3934/dcds.2022075

\bibitem{Cazenave-Escobedo} Cazenave, T., Dickstein, T.  and Escobedo, M.: A
semilinear heat equation with concave-convex nonlinearity, {\em Rendiconti di
Matematica}, Serie VII, {\bf 19} (1999), 211-242.

\bibitem {Constantin2010nagumo} Constantin, A.: On Nagumo's theorem, {\em Proc. Japan Acad.}, (Ser A), (2010), 41--44.

\bibitem{Constantin2023uniqueness} Constantin, A.: A uniqueness criterion for ordinary differential equations, {\em J. Differential Equations}, {\bf 342} (2023), 179--192.

\bibitem{CrandallLiggett} Crandall, M.G. and Liggett, T.M.: Generation of semi-groups of nonlinear tranformarion on general Banach spaces, {\em Am. J. of Math}, {\bf (93)} (1971), 265-295.

\bibitem{da1994stochastic} Da Prato, G. and Zabczyk, J.: {\em Stochastic Equations in Infinite Dimensions}, {\bf 152}, 2014 (see also First Edition 1992).

\bibitem{Fereisel} Feireisl, E. A note on uniqueness for parabolic problems
with discontinuous nonlinearities, {\em Nonlinear Analysis: Theory, Methods \&
Applications}, {\bf 16}(11) (1991), 1053-1056.


\bibitem{FujitaWatanabe1968}
Fujita, H.  and Watanabe, S.: On the uniqueness and non-uniqueness of solutions of initial value problems for some quasi-linear parabolic equations,
{\em Commun. Pure and Appl. Anal.}, {\bf 22} (3) (2023), 686--735.

\bibitem{Diazlarge2023} Díaz, G.: Large solutions of elliptic semilonear equarions non--degenerate nera the boundary, {\em Communications on Pure on Appied Mathematics}, {\bf 22} (3) (2023), 631--652

\bibitem{Diaz26stochastic} Díaz, G. and Diaz, J.I.: Stochastic diffusive energy balance climate model with a multiplicative noise modeling the Solar variability. To appear in {\em Journal Pure and Applied Functional Analysis}.

\bibitem{DPitman} D\'{\i}az, J.I.: {\em Nonlinear Partial Differential Equations
and Free Boundaries}, Pitman, London, 1985.

\bibitem{DiazEscorial} D\'{\i}az, J.I.: Mathematical analysis of some
diffusive energy balance climate models. In {\em Mathematics, Climate and
Environment} (J. D\'{\i}az and J. -L. Lions, eds.) Masson, Paris, 28--56,
1993.

\bibitem{D-Monotonicity} D\'{\i}az, J.I.: New applications of monotonicity
methods to a class of non-monotone parabolic quasilinear sub-homogeneous
problems. Journal Pure and Applied Functional Analysis. (Special Issue
dedicated to Ha\"{\i}m Brezis), 5 4 (2020), 925-949.



\bibitem{DiazTello}Díaz, J.I. and Tello, L.: On a nonlinear parabolic problem on a Riemannian manifold without boundary arising in Climatology, {\em Collectanea Mathematica}, Volum {\bf L}, Fascicle 1, (1999), 19-51. 


\bibitem{D-Hernandez-Ilasov} D\'{\i}az, J.I., Hern\'{a}ndez, J. and 
Il'yasov, Y.: Flat solut\'{\i}ons of sorne non-L\'{\i}psch\'{\i}tz
autonomous semilinear equations may be stable for N {\ bf 126}3, {\em Chinese
Ann. Math.}, Series B {\bf 38} (2017), 345-378.

\bibitem{Diaz-Hetzer} D\'{\i}az, J.I. and Hetzer, G.: A quasilinear functional
reaction-diffusion equation arising in {\em Climatology, Equations aux derivees
partielles et applications: Articles dedies a Jacques Louis Lions}, Gautier
Villards, Paris (1998), 461--480.

\bibitem{DiazSaa} Díaz. J.I. and Saa, J.E.: Existence et unicité de solutions positives pour certaines équations elliptiques quasilinéaires. {\em Comptes Rendus Acad. Sc. París}, t. {\bf 305}, Série I, (1987), 521‑524.

\bibitem{DVrabie} D\'{\i}az, J.I. and I. I. Vrabie I.I.: Existence for reaction%
-diffusion systems. A compactness method approach, {\em Journal of
Mathematical Analysis and Applications}, {\bf 188}, 2 (1994), 521-540.

\bibitem{Levy}Lévy, P.: {\em Processus stochastiques et mouvement brownien}, Gauthier-Villars $\&$  Cie, Paris, 1965.

\bibitem{LiuRockner} Liu, W. and R\"{o}ckner, R.: {\em Stochastic Partial
Differential Equations: An Introduction}, Springer, 2015.


\bibitem{Nagumo}Nagumo, M. : Eine hinreichende Bediingung für die Unität der Lösung von Differentialgleichunge erster Ordnung, {\em Jpn. J. Math.} {\bf 3} (1926) 107–112.

\bibitem{Osgood1898} Osgood, W.F.: Beweis der Existenz einer Lösung der Differentialgleichung $dy/dx=f(x,y)$ ohne Hinzunahme der Cauchy-Lipschitz'schen Bedingung,
{\em Monatshefte für Mathematik und Physik}, 9 (1898), 331--345.


\bibitem{Pao1992} Pao, C.V.: {\em Nonlinear Parabolic and Elliptic Equations}, Plenum Press, New York, 1992.

\bibitem{Stone} Stone, P.H.: A simplified radiative - dynamical model for the
static stability of rotating atmospheres, {\em Journal of the Atmospheric
Sciences}, {\bf 29}, No. 3, (1972) 405-418.

\bibitem{Taniguchi1992}
Taniguchi, T.: Successive approximations to solutions of stochastic differential equations,
{\em J. Differential Equations}, 96 (1992), 152--169.


\bibitem{TemamLibro} Temam, R.: {\em Infinite-Dimensional Dynamical Systems in
Mechanics and Physics}. Springer New York, 1997.

\bibitem{Vrabie} Vrabie, I. I.: Compactness Methods for Nonlinear
Evolutions, Second Edition, {\em Pitman Monographs and Surveys in Pure and
Applied Mathematics} {\bf 75}, Longman 1995.

\bibitem{Yamada1981} 
Yamada, T.:  On the successive approximation of solutions of stochastic differential equations,
{\em J. Math.  Kyoto Univ}, {	\bf 21} (3) (1981), 501--515.




\end{thebibliography}
\end{document}